\documentclass{amsart}
\usepackage{xy, enumerate, pstcol, amssymb, amsfonts, amsbsy, amsthm, amsmath, amscd, epsfig, latexsym, mathtools, stmaryrd, epic, eepic, eucal, multicol, fancyhdr, graphicx,tikz,tikz-cd}

\DeclareMathAlphabet{\mathpzc}{OT1}{pzc}{m}{it}

\newtheorem{theorem}{\bf Theorem}[section]

\newtheorem{lemma}[theorem]{\bf Lemma}

\newtheorem{corol}[theorem]{\bf Corollary}

\newtheorem{question}[theorem]{\bf Question}
\theoremstyle{definition}
\newtheorem{defi}[theorem]{\bf Definition}
\newtheorem{rmk}[theorem]{\bf Remark}

\newtheorem*{theoremola}{Theorem \ref{thm:orbla}}

\newcommand{\A}{\mathbb A}

\newcommand{\C}{\mathbb C}
\newcommand{\Ccal}{\mathcal C}

\newcommand{\Ecal}{\mathcal E}
\newcommand{\bE}{\mathbb E}

\newcommand{\Lcal}{\mathcal L}

\newcommand{\Ocal}{\mathcal O}
\newcommand{\Pro}{\mathbb P}
\newcommand{\Pcal}{\mathcal P}
\newcommand{\Q}{\mathbb Q}

\newcommand{\Ucal}{\mathcal U}
\newcommand{\Xcal}{\mathcal X}
\newcommand{\Ycal}{\mathcal Y}
\newcommand{\Z}{\mathbb Z}

\newcommand{\Cont}{\text{Cont}}
\newcommand{\rk}{\text{rk}}

\newcommand{\Adm}[2]{Adm_{#2}\left({ #1}\right)}

\input xy
\xyoption{all}

\title{All the $\lambda_1$'s on  cyclic admissible covers}
\author{Renzo Cavalieri}
\author{Bryson Owens}
\author{Seamus Somerstep}

\address{Renzo Cavalieri, Department of Mathematics, Colorado State University, Fort Collins, CO}
\email{renzo@math.colostate.edu}
\address{Bryson Owens, Department of Mathematics, Statistics and Computer Science, University of Illinois, Chicago, IL}
\email{bowens21@uic.edu}
\address{Seamus Somerstep, Department of Statistics, University of Michigan, Ann Arbor, MI}
\email{smrstep@umich.edu}

\begin{document}

\begin{abstract} 
We compute the degree of Hurwitz-Hodge classes $\lambda_1^e$ on one dimensional moduli spaces of cyclic admissible covers of the projective line. We also compute the degree of the 
the first Chern class of the Hodge bundle $\lambda_1$ for all one dimensional moduli spaces. In higher dimension, we  express the divisor class $\lambda_1$ as a linear combination of $\psi$ classes and boundary strata.  
\end{abstract}

\maketitle

\section{Introduction}

A cyclic cover is the quotient map of a curve $C$ by the effective action of a cyclic group $\Z/ d\Z$; for example the projection to the $x$-axis from the curve $\{y^d = p(x)\}\subset \A^2$ gives an affine model determining by the Riemann existence theorem a cyclic cover of $\Pro^1$. Since the complex structure of a cyclic cover of $\Pro^1$ is essentially determined by the location and the monodromies of its branch points, families of cyclic covers are  among the most classical ways to construct subvarieties of moduli spaces of curves.

The perspective of Hurwitz spaces, together with the Harris-Mumford   admissible covers compactification in \cite{hm:kd}, proposes to study families of cyclic covers with fixed discrete invariants (genus, degree, and monodromy data around the branch points) as standalone moduli spaces, connected to moduli spaces of curves by natural source and branch morphisms. 

With the development of the language of orbifolds and stacks, \cite{acv:ac} interpret the normalization of the Harris-Mumford space as a (connected component of a) smooth stack of {\it twisted stable maps} from some orbifold modification of the base curves to a quotient stack $[pt./G]$. We use the terminology {\it moduli space of cyclic admissible covers} $\Adm{m_1, \ldots, m_n}{d}$ to denote this smooth stack rather than the Harris-Mumford version.

The branch and source morphisms from spaces of admissible covers create a correspondence that connects the geometry of moduli spaces of higher genus (cover) curves with the combinatorics of the configuration of branch points, thus making explicit and accessible certain geometric information, see \cite{bp:drg2}.
We illustrate this philosophy by studying the Hodge bundle $\bE$, whose fiber over a moduli point consists of the global sections of the relative dualizing sheaf of the cover curve, in terms of combinatorial information from $\overline{M}_{0,n}$. Further, the cyclic action on the curve induces an action on the Hodge bundle.
The language of  twisted stable maps is very well attuned to study the Chern classes of subrepresentations $\bE_e$ of the Hodge bundle, called {\it Hurwitz-Hodge classes} in \cite{bgp:crc}. Hurwitz-Hodge integrals, i.e. intersection numbers of these classes, are used in the computations of orbifold Gromov-Witten invariants \cite{cc:c3z3, bg:crchhi}, and were a key tool in the development and study of the crepant resolution conjecture in the Hard Lefschetz case \cite{bg:crchl}.

While the orbifold version of the Grothendieck-Riemann-Roch theorem \cite{t:grr} provides a powerful technique to compute individual Hurwitz-Hodge integrals \cite{z:choi}, it conceals the rich algebraic and combinatorial structure that  Hurwitz-Hodge classes have when considered in families.
 
The main result of this article  highlights this structure in the case of the first Chern class of the bundles $\bE_e$, that we denote by $\lambda_1^e$. We compute the degree of these classes on all one dimensional moduli spaces of cyclic admissible covers.

\begin{theorem} \label{thm:orbla} Let $d$ be a positive integer,  $m_1, m_2, m_3, m_4$ a monodromy datum as in Definition \ref{def:mondat}, and $e$ an integer  between $0$ and $d-1$. Assume, without loss of generality, that 
$$
\left\langle\frac{em_1}{d}\right\rangle \leq\left\langle\frac{em_2}{d}\right\rangle  \leq\left\langle\frac{em_3}{d}\right\rangle 
\leq \left\langle\frac{em_4}{d}\right\rangle. 
$$
The degree of the orbifold class $\lambda^e_1$ on the one-dimensional space of degree $d$ cyclic admissible covers of a rational curve with monodromies $m_i$ is given by the following  formula:
\begin{equation}\label{eq:onedimorb}
\int_{\Adm{m_1, m_2, m_3, m_4}{p}}\lambda^e_1 =
\left\{
\begin{array}{cl}
\min\left\{\frac{1}{d}\left\langle\frac{em_1}{d}\right\rangle, \left(\sum_{i=1}^4\left\langle\frac{em_i}{d}\right\rangle\right) -1\right\} & \mbox{if}\ \left\langle\frac{em_1}{d}\right\rangle +\left\langle\frac{em_4}{d}\right\rangle \leq 1 \\ & \\
\min\left\{\frac{1}{d}\left(1-\left\langle\frac{em_4}{d}\right\rangle\right), 3- \left(\sum_{i=1}^4\left\langle\frac{em_i}{d}\right\rangle\right) \right\} & \mbox{if}\ \left\langle\frac{em_1}{d}\right\rangle +\left\langle\frac{em_4}{d}\right\rangle > 1
\end{array}
\right.
\end{equation}
\end{theorem}

The next result computes the degree of the class $\lambda_1$, the first Chern class of the full Hodge bundle $\bE$, for all one dimensional moduli spaces of cyclic admissible covers.

\begin{theorem}
\label{thm:dimonenotorb}
Let $d$ be a positive integer and $m_1, m_2, m_3, m_4$ a monodromy datum. The degree of the class $\lambda_1$ on the one-dimensional space of degree $d$ cyclic admissible covers of a rational curve with monodromies $m_i$ is given by the following  formula, indexed by the power set of $[4] = \{1,2,3,4\}$.
\begin{equation}\label{eq:onedim}
\int_{\Adm{m_1, m_2, m_3, m_4}{d}}\lambda_1 = \frac{1}{24 d^2} \left( \sum_{I\in \Pcal([4])}(-1)^{|I|} {\gcd}^2\left(\sum_{i\in I}m_i, d\right)\right).
\end{equation}
\end{theorem}

In the case of the full Hodge bundle, the class $\lambda_1$ can be described as a linear combination of boundary strata, $\psi$ classes and the class $\kappa_1$; these are considered the standard generators for the tautological ring of the moduli spaces of curves. The general formula for $\lambda_1$ is a natural generalization of the one-dimensional case.

\begin{theorem} \label{thm:gf} Let $d$ be a positive integer and $m_1, \ldots, m_n$ be integers  with $0\leq m_i <d$. The  class $\lambda_1$ on the space $\Adm{m_1, \ldots, m_n  }{d}$ of cyclic admissible covers of a rational curve with monodromies $m_i$  is equivalent to the following tautological expression:
\begin{equation} \label{eq:gf}
 \lambda_1 = \frac{1}{24 d} \left( \sum_{J\in \Pcal([n])} {\gcd}^2\left(\sum_{j\in J}m_i, d\right) \Delta_J \right),   
\end{equation}
 where
\begin{itemize}
\item for $2\leq |J|\leq n-2$, $\Delta_J$ denotes the boundary divisor generically parameterizing one-nodal curves with the branch points labelled by $J$ sitting on one component, and the branch points in $J^c$ on the other;
\item for $J = \{j\}, [n]\smallsetminus \{j\}$, $\Delta_J:= -\psi_j$;
\item for $J  = \phi, [n]$, $\Delta_J:= \kappa_1$.
\end{itemize}

\end{theorem}

Theorems \ref{thm:orbla} and \ref{thm:dimonenotorb} are proved using Atyiah-Bott localization, following a strategy introduced for ordinary Hodge integrals in \cite{fp:hiagwt} and imported to the setting of admissible covers in \cite{r:adm,r:tqft}. The idea is to set-up a vanishing auxiliary integral on a moduli spaces of  admissible covers of a parameterized $\Pro^1$: these are moduli spaces that admit a torus action, and hence the vanishing integral can be evaluated by restricting it to fixed loci for the torus action, giving rise to a relation among Hurwitz-Hodge integrals. By choosing the auxiliary integral carefully, one can arrange for the principal part of the relation to contain the degree of $\lambda_1$ (or $\lambda_1^e$), and for all other terms to be either zero dimensional moduli spaces, or functions  of the degree of a $\psi$ class, which is well-known.

Since the full Hodge bundle is the direct sum of its subrepresntations, we have the relation 
\begin{equation}\label{eq:laorbla}
 \lambda_1 =\sum_{e=1}^{d-1} \lambda_1^e,   
\end{equation}

hence in principle one may deduce Theorem \ref{thm:dimonenotorb} as a corollary of Theorem \ref{thm:orbla}. While it is certainly straightforward to do so for any individual case, to obtain a general result we found it in the end more convenient  to approach the non-orbifold computation independently. One may therefore use \eqref{eq:laorbla} and the two theorems to deduce an elementary, albeit mysterious to us, arithmetic identity. 

The formula for the degree of $\lambda_1$ from Theorem \ref{thm:dimonenotorb} naturally generalizes to higher dimensional moduli spaces to produce a {\it graph formula}: a representation of $\lambda_1$ in terms of combinatorially decorated strata classes, in a fashion similar to \cite{jppz:dr}. Once one is able to guess a graph formula, one may prove it simply by computing intersection numbers with all boundary curves, using the fact that boundary curves generate $A_1(\Adm{m_1, \ldots, m_n}{d})$.
The study of a graph formula description of $\lambda_1$ was initiated in \cite{pete,champs}, which independently computed the hyperelliptic case. The second and third authors studied the $d=3$  case in \cite{cos:deg3}.
Theorem \ref{thm:gf} concludes this analysis for all spaces of cyclic covers of $\Pro^1$.
While we are currently unable to produce a graph formula for the orbifold classes $\lambda_1^e$, our observations led us to the following question.
 
\begin{question}
With notation as in Theorem \ref{thm:gf}, the following is a graph formula for the classes $\lambda_1^e$ in the $4$-pointed case: 
$$
\lambda^e_1 = \frac{1}{2} \left( \sum_{I\in \Pcal([4])} \min \left\{0,1 - \sum_{i\in I} \left\langle\frac{em_i}{d}\right\rangle \right\}\Delta_I \right).
$$
Is there a natural generalization that yields a graph formula in the general case?
\end{question}
One of the goals of this manuscript is to  communicate both to the readers versed in a more classical algebraic geometric language, as well as with the readers steeped in orbifold technology. We feel that there is potential for fruitful interactions between the two communities, and dedicate a good part of the background section in recalling (albeit briefly) the main connections between the two languages.

\subsection{Acknowledgments} We thank Rachel Pries and John Voight for interesting conversations related to this project. The first author acknowledges  support from the Simons collaboration grant 420720 and NSF grant DMS 2100962. The second and third authors received support from the Mathematics Dept. of Colorado State University.

\section{Background}
In this section we collect some background needed for the computations in the later sections. While there is no pretense to make this work self-contained, we provide basic information for translating between the languages of covers and orbifold maps, and suggest references for readers interested in more details. 
\subsection{Orbifolds}

Orbifolds (or stacks in the algebraic category) are geometric objects generalizing the notion of orbit spaces to the case of non-free group actions. We refer the reader to  \cite{ruanadem},\cite{Fantechi} for a comprehensive introduction, and limit ourselves to recalling the aspects of the theory that are relevant to this manuscript.

\subsubsection*{Global quotient orbifolds.}
Given a space $X$ and a group $G$ acting on it, one denotes by $[X/G]$ the orbifold quotient of $X$ by $G$ (see \cite{fg:ocgq}). The most useful way to define this concept is through its functor of points; in simple terms, this means that one gives geometric structure to $[X/G]$ by describing the functions to it.

\begin{defi} \label{def:go}
A function $f: B\to [X/G]$ consists of a 
pair:
\begin{itemize}
    \item $\pi_f:E\to B$, a principal $G$-bundle over $B$;
    \item $F: E\to X$, a $G$-equivariant map.
\end{itemize}  
\end{defi}
Definition \ref{def:go} produces a fiber diagram
\begin{equation}
    \xymatrix{
    E \ar[d]_{\pi_f}\ar[r]^F& X\ar[d]^\pi \\
    B \ar[r]^f& [X/G]
    }
\end{equation}
analogous to the universal property of orbit spaces in the case of a free action.

One may still think of geometric points of $[X/G]$ as orbits $[x]$ of points of $X$ under the $G$ action; each point $[x]$ comes with the additional information of an {\it isotropy group} $G_x$, isomorphic to the stabylizer of any point $x\in [x]$.

In the extreme case when $X = pt.$, we denote the quotient orbifold $[pt./G]$ by $BG$, and call it the  classifying space for principal $G$-bundles (as the datum of the map $F$ becomes in this case trival). This orbifold consists of a point with isotropy group $G$.
A {\it line bundle} on $BG$ consists of a one dimensional representation of $G$.

\begin{defi}
For $d$ a non-negative integer, and $0\leq e < d-1$ we denote by $\Ocal_e$ the line bundle on $B(\Z/d\Z)$ corresponding to the representation
$$
[1]\cdot z = \exp\left({e\frac{2\pi i}{d}}\right)z.
$$
For $j\in \Z$, we denote by $L_j$ the line bundle on $B\C^\ast$ corresponding to the representation
$$
\alpha \cdot z = \alpha^jz.
$$
\end{defi}
\begin{rmk} The following observations will be useful later:
\begin{itemize}
    \item we denote by $\Ocal_e$ and $L_j$ also line bundles that are pulled back from $BG$ to other spaces. These consist of trivial line bundles with a $G$-action;
    \item in the case of $B\C^\ast$, we  need to work with $\Q$-divisors;  the index $j$ in $L_j$ will then be a rational number.
\end{itemize}
\end{rmk}

\subsubsection*{Twisted curves.}
Global quotient orbifolds are the local models for the construction of more general orbifolds. A treatment of one-dimensional orbifolds and orbifold line bundles is in \cite{j:egwsc}. A {\it twisted curve} $\Ccal$ is obtained from a Riemann surface $C$ by replacing a finite number of disjoint open discs of $C$ with (open sets of) global quotient orbifolds $[\C/ (\Z/d\Z)]$; the curve $\Ccal$ and its so called coarse moduli space $C$ have the same geometric points, but the twisted curve has a finite number of points, called {\it twisted points}, with non-trivial, cyclic isotropy groups.

A line bundle $\Lcal$ on a twisted curve $\Ccal$ contains the information of a representation of the isotropy group $G_x$ over every twisted point, describing a lift of the group action to the fiber of the line bundle. If $G_x = \Z/d\Z$ and $[1]\cdot w = \mbox{e}^{k\frac{2\pi i}{d}}w$, then the rational number $k/d$ is called the {\it age} of $\Lcal$ at $x$. The age of a line bundle at the twisted points contributes to the orbifold version of the Riemann-Roch theorem.

\begin{theorem}[Orbifold Riemann-Roch,  \cite{agv:gwtodms}, 7.2.1] \label{thm:orroch}
Let $\Ccal$ be a twisted curve, and $\Lcal$ a line bundle on it; then:
$$
h^0(\Ccal,\Lcal)- h^1(\Ccal,\Lcal) = \deg (\Lcal) +1 -g_\Ccal -\sum_{x\in \Ccal} age_x(\Lcal).
$$
\end{theorem}
\subsubsection*{Maps of orbifolds}
A map of orbifolds $f:\Xcal \to \Ycal$ contains additional information with respect to the function on the coarse spaces $f:X\to Y$. For every point $x\in \Xcal$, $f$ determines a group homomorphism $\phi_{f,x}: G_x \to G_{f(x)}$.
The map $f$ is called {\it representable} if all the group homomorphisms $\phi_{f,x}$ are injective.

\subsection{Cyclic admissible covers and twisted stable maps}

Admissible covers were introduced in \cite{hm:kd} to compactify the Hurwitz spaces. In \cite{acv:ac}, the authors show that connected components of the smooth stack of twisted stable maps to $BS_d$  realize the normalization of  spaces of admissible covers. We now introduce these spaces in the case where the target is $B(\Z/d\Z)$, and adopt the convention of calling {\it admissible covers} the smooth stack of twisted stable maps.

Let $\Ccal$ be a twisted curve whose coarse space  $C\cong \Pro^1$ is smooth. By Definition \ref{def:go}, a map $f: \Ccal \to B \Z/d\Z$ consists of a principal $\Z/d\Z$-bundle $\Ecal\to \Ccal$. Passing to the corresponding coarse spaces one obtains a cyclic cover $\pi_f: E\to C$, that ramifies only over the (image of the) twisted points of $\Ccal$. To be precise, if
$
\phi_{x,f}([1]) = m\in \Z/d\Z, 
$
then $\pi_f^{-1}(x)$ consists of $\gcd(m,d)$ points of ramification order $d/\gcd(m,d)$;
further, if one chooses a generic point $x_0\in C$ as a base point and a labelling of its $d$ inverse images that is compatible with the cyclic group action on $E$, then the image of a small loop around $x$ in the monodromy representation associated to the cover $\pi_f$ is precisely $m$. Hence we say that $x$ is a point of monodromy $m$ for $\pi_f$.

\begin{defi} \label{def:mondat}
Let $d$ be a positive integer. A set of integers $m_1, \ldots, m_n$ with $0\leq m_i <d$ are called {\bf a monodromy datum} for degree $d$ cyclic admissible covers of rational curves if
$
   \sum_{i=1}^n m_i = 0  \mod d
$
and
$
\gcd(m_1, \ldots, m_n,d) = 1.
$
\end{defi}
\begin{rmk}
The first condition in Definition \ref{def:mondat} guarantees that admissible covers exist; the second one requires them to be connected.
\end{rmk}
\begin{defi}
Given a monodromy datum $m_1, \ldots, m_n$, the moduli space of admissible covers $\Adm{m_1, \ldots, m_n}{d}$ parameterizes representable maps $\Ccal\to B(\Z/d\Z)$ such that:
\begin{itemize}
    \item $C$ is a rational, nodal curve;
    \item $\Ccal $ has exactly $n$ twisted  points labeled $x_1, \ldots, x_n$ in the smooth locus of $C$, of monodromies $m_1, \ldots, m_n$;
    \item the nodes of $\Ccal$ may be twisted; in that case, the  two shadows of a node in the normalization of $\Ccal$ have opposite monodromies.
\end{itemize}
\end{defi}

As before, a map $\Ccal\to B(\Z/d\Z)$ corresponds to a cyclic cover $E\to C$ of the rational nodal curve $C$. The nodes of $E$ are precisely the inverse images of nodes of $C$, and the ramification orders of pairs of shadows of any node must match.  

Spaces of admissible covers have universal morphisms denoted as in the following diagram:
\begin{equation} \label{eq:acuc}
    \xymatrix{\Ucal_E\ar[d] \ar[rr]^F \ar@/_5pc/[ddr]_\Pi& & pt.\ar[d]\\
    \Ucal_C  & \ar[l] \Ucal_\Ccal \ar[r]^f \ar[d]_\pi & B(\Z/d\Z)\\
     & \Adm{m_1, \ldots, m_n}{d}& }
\end{equation}

By the Riemann-Hurwitz formula, the genus of the cover curve $E$ is
\begin{equation}
    g = 1+\frac{(n-2)d-\sum_{i=1}^n\gcd(m_i,d)}{2}.
\end{equation}

There are two natural morphisms from the space of admissible covers. The {\it source} morphism
\begin{equation}
    s: \Adm{m_1, \ldots, m_n}{d}\to \overline{M}_{g}
\end{equation}
remebers the cover curve $E$. The {\it branch} morphism
\begin{equation}
    br: \Adm{m_1, \ldots, m_n}{d}\to \overline{M}_{0,n}
\end{equation}
records the base curve $C$ together with the images of the $n$-twisted points of $\Ccal$. The branch morphism is a bijection on geometric points, but every point of $\Adm{m_1, \ldots, m_n}{d}$ has an order $d$ cyclic isotropy group, and hence
$
\deg (br)  = \frac{1}{d}$.

\subsubsection*{Boundary stratification}
Boundary strata for the space $\Adm{m_1, \ldots, m_n}{d}$ are in canonical bijection with boundary strata of $\overline{M}_{0,n}$, and they can therefore be indexed by the dual graphs of the base curves.

Any subset $I\in \Pcal([n])$  with $ 2\leq |I|\leq n-2$ identifies a boundary divisor $\Delta_I$, whose set of points is isomorphic to the product
\begin{equation}
  \Adm{\{m_i\}_{i\in I}, \left[-\sum_{i\in I}m_i\right]_d}{d}\times \Adm{\{m_i\}_{i\in I^c}, \left[\sum_{i\in I}m_i\right]_d}{d};   \end{equation}
as a stack, however $\Delta_I$ is isomorphic to a fiber product over $B \Z/d\Z$ of the two factors above. This causes, when integrating along $\Delta_I$, a  factor of $d$ often referred to as the gluing factor (see \cite[Section 1.6]{cc:c3z3} for a discussion). To remember this, we abuse notation and write
\begin{equation}
 \Delta_I\cong d\cdot \Adm{\{m_i\}_{i\in I}, \left[-\sum_{i\in I}m_i\right]_d}{d}\times \Adm{\{m_i\}_{i\in I^c}, \left[\sum_{i\in I}m_i\right]_d}{d}.   \end{equation}

We are especially interested in one-dimensional boundary strata, or boundary curves, in $\Adm{m_1, \ldots, m_n}{d}$. Their dual graphs are trees that have a unique vertex $v$ of valence $4$, and all other vertices trivalent. Removing $v$ the set of indices is partitioned into four sets $X,Y,Z,W$. The rational equivalence class of a boundary curve depends only on such partition, and hence we denote a boundary curve class by $C_{(X,Y,Z,W)}$.
We observe that, keeping in account gluing factors as well as automorphism factors coming from the zero dimensional moduli spaces in the product expression of a boundary curve, in the end one has:
\begin{equation}
   C_{(X,Y,Z,W)}\cong \Adm{\left[\sum_{x\in X} m_x\right]_d,\left[\sum_{y\in Y} m_y\right]_d,\left[\sum_{z\in Z} m_z\right]_d,\left[\sum_{w\in W} m_w\right]_d  }{d}.
\end{equation}

\subsubsection*{Parameterized admissible covers}
Given a monodromy datum, we denote by 
$$
\Adm{\Pro^1| m_1, \ldots, m_n}{d}
$$
the space of {\it admissible covers of a parameterized $\Pro^1$} (\cite{r:tqft}); using orbifold language these are twisted stable maps of degree $(1,0)$ to $\Pro^1\times B(\Z/d\Z)$. In terms of the geometry of the covers $E\to C$, this means that one component  $\Pro^1\subseteq C$ is chosen, and for two covers to be isomorphic, the isomorphism of the base curves must restrict to the identity of the special component $\Pro^1$.
A universal diagram analogous to \eqref{eq:acuc} holds for spaces of parameterized admissible covers, where one takes the cartesian product with $\Pro^1$ of every space in the rightmost column.

Given a parameterized admissible cover
$f: \Ccal \to \Pro^1\times B(\Z/d\Z)$ and a bundle $\Lcal = \Ocal_{\Pro^1}(n)\otimes \Ocal_e$, the cohomology groups of $f^\ast\Lcal$ may be described in terms of the geometry of the cover $E\to C\to \Pro^1$. Unraveling the appropriate orbifold definitions (as  in \cite{bgp:crc}), one obtains
\begin{equation}\label{eq:unravel}
    H^i(\Ccal, f^\ast\Lcal) = \left(H^i(E, F^\ast\Ocal_{\Pro^1}(n) ) \right)_{d-e}, 
\end{equation}
where $F$ is the composition of the two maps above, and the subscript $(d-e)$ denotes the subrepresentation of the cohomology group corresponding to the character $(d-e)$.


\subsection{Tautological classes}

Intuitively, tautological classes on a family of moduli spaces are elements of the Chow (or cohomology) ring that are constructed using the intrinsic  geometry of the objects parameterized (e.g. the structure sheaf or the dualizing sheaf), via a series of operations that involve push-forwards and pull-backs via tautological morphisms. We refer the reader to \cite{v:mscgw} for a proper introduction to the subject, and here focus on introducing the objects and properties that we need.

\subsubsection*{Chern classes of bundles} \label{sec:chern}
See \cite{f:it} for a less skeletal introduction to the subject. Given a vector bundle $E\to X$ of rank $r$, the {\it total Chern class} of $E$ is a Chow class
\begin{equation}
    c(E):= 1+ c_1(E)+\ldots+c_r(E),
\end{equation}
where $c_i(E)\in A^i(X)$ is called the $i$-th Chern class.
Perhaps the single most important formal property in this theory is that the total Chern class is multiplicative with respect to extensions, in the sense that given a sequence $0\to F\to E\to Q\to 0$, one has $c(E) = c(F)c(Q)$; this readily implies that the first Chern class is additive, i.e. $c_1(E) = c_1(F)+c_1(Q)$.

Given a bundle $E$ of rank $r$, its {\it Chern roots} $\alpha_1, \ldots, \alpha_r$ are graded symbols of degree one with the defining property that the $i$-th Chern class of $E$ is the $i$-th elementary symmetric function in the Chern roots.

Chern roots allow to treat arbitrary bundles as if they split as the direct sum of line bundles, and are hence useful tools to compute Chern classes. For example, if $E$ is as above and $L$ a line bundle, then the Chern roots of the bundle $E\otimes L$ are $\alpha_i+c_1(L)$ and this fact allows to readily  compute the Chern classes of the tensor product in terms of the Chern classes of the factors. As an application, we show a computation needed in Section \ref{sec:non-orb}.

\begin{lemma}\label{lem:chern}
$$
c_1(E^{\oplus n}) = n c_1(E)
$$
$$
c_2(E^{\oplus n}) = n c_2(E) + {{n}\choose{2}} c_1^2(E).
$$
\end{lemma}

\begin{proof}
These statements follow immediately from the multiplicativity of total Chern classes.
\end{proof}

\subsubsection*{Hodge bundles.} 
The main object of study of this work is the first Chern class of the Hodge bundle and its orbifold variants on spaces of admissible covers, which we now introduce.

\begin{defi} Consider a space of admissible covers $\Adm{m_1, \ldots, m_n}{d}$ and the universal morphisms from diagram \eqref{eq:acuc}. We define:
$$
\bE:= (R^1\Pi_\ast F^\ast(\Ocal))^\vee.
$$
\end{defi}
One observes that $\bE = s^\ast(\bE)$, i.e. the Hodge bundle on the space of admissible covers is the pull-back of the homonymous bundle on the moduli space of curves via the source morphism. Its fibers over a general point $\Ccal \to B \Z/d\Z$ corresponding to a smooth cover $E\to C$ may be identified with the vector space of holomorphic one-forms on $E$. It follows that the rank of $\bE$ is equal to the genus of $E$.

\begin{defi} Consider a space of admissible covers $\Adm{m_1, \ldots, m_n}{d}$ and the universal morphisms from diagram \eqref{eq:acuc}. We define:
$$
\bE_e:= (R^1\pi_\ast f^\ast(\Ocal_{e}))^\vee.
$$
\end{defi}

The rank of the bundle $\bE_e$ is computed using the orbifold Riemann-Roch theorem:
\begin{equation}\label{eq:rankhodgeorb}
    \rk(\bE_e) = \rk ((\bE^\vee)_{d-e}) = \rk(R^1\pi_\ast f^\ast\Ocal_e)= h^1(C, f^\ast(\Ocal_e)) = -1 + \sum_{i=1}^n age_{x_i}(f^\ast(\Ocal_e)) = -1 + \sum_{i=1}^n \left\langle \frac{em_i}{d}\right\rangle
\end{equation}

The notation $\bE_e$ follows from the interpretation of its fibers in terms of the geometry of the covers $E\to C$. Applying \eqref{eq:unravel}, one sees that the fiber of $\bE_e$ over a general moduli point corresponding to a cover  $E\to C$ equals the $e$-subrepresentation of the space of holomorphic one forms on $E$.

It follows  that \begin{equation}
\bE = \bigoplus_{e=0}^{d-1} \bE_e.    
\end{equation}

We need to work with subrepresentations of the dual of the Hodge bundle as well. The natural way to induce an action on a dual space implies:
\begin{equation} \label{eq:dual}
    (\bE^\vee)_{e} \cong (\bE_{d-e})^\vee.
\end{equation}

We recall the $G$-Mumford relation, introduced in \cite{bgp:crc}:
\begin{equation} \label{eq:GMumf}
    c(\bE_e \oplus  (\bE^\vee)_{e}) = 1.
\end{equation}

\begin{defi}
We define:
$$
\lambda_1:= c_1(\bE),
$$
$$
\lambda^e_1:= c_1(\bE_e).
$$
\end{defi}

\begin{lemma}\label{lem:sym}
For any choice of positive integers $d,e$, with $0\leq e<d$, and monodromy datum $m_1, \ldots, m_n$,
\begin{equation} 
    \lambda_1^e = \lambda_1^{d-e}.
\end{equation}
\end{lemma}
\begin{proof}
We temporarily denote by $\hat\lambda^e_1$ the first Chern class of the $e$-eigenbundle of the dual of the Hodge bundle on the space $\Adm{m_1, \ldots, m_n}{d}$.
Equation \eqref{eq:dual} implies that $\hat\lambda_1^e = -\lambda_1^{d-e}$; the $G$-Mumford relation \eqref{eq:GMumf}  implies $\lambda_1^e+\hat\lambda_1^e =0$. Combining the two equations the statement follows.
\end{proof}
\subsubsection{Psi classes.}
We recommend \cite{k:pc} for an  introduction to $\psi$ classes on moduli spaces of curves. 

\begin{defi}
Consider the moduli space $\Adm{m_1, \ldots, m_n}{d}$. For $1\leq i\leq n$
we denote by $\psi_i$ the pullback
$br^\ast(\psi_i)$, where we assume the notion of $\psi$ classes on $\overline{M}_{0,n}$.
\end{defi}
To attach some meaning to this definition for readers who are completely unfamiliar with $\psi $ classes, the class $\psi_i$ is the first Chern class of a line bundle on  $\Adm{m_1, \ldots, m_n}{d}$ whose fiber over a moduli point $E\to C$ is canonically identified with the cotangent line of $C$ at the $i$-th branch point.

By the projection formula, we have
\begin{equation}
    \int_{\Adm{m_1, \ldots, m_n}{d}} \prod \psi_i^{k_i} = \frac{1}{d}\int_{\overline{M}_{0,n}} \prod \psi_i^{k_i} =\frac{1}{d}{{n-3}\choose{k_1, \ldots, k_n}}.
\end{equation}


\subsection{Atiyah-Bott localization}
We give a brief account of localization and develop some details geared to our application of it. We follow the  language and notations in \cite[chapters $4$ and $27$]{clay:ms}. A complete reference for this technique for moduli spaces of maps from orbifold curves is \cite{l:localize}.

Consider the  one-dimensional algebraic torus $\mathbb{C}^\ast$, and recall that the $\mathbb{C}^\ast$-equivariant Chow ring of a point is a polynomial ring in one variable:
$$A^\ast_{\C^\ast}(\{pt\},\mathbb{C})= \mathbb{C}[t], $$
with $t = c_1(L_1).$

Let $\C^\ast$ act on a smooth, proper stack $X$, denote by $i_k:F_k\hookrightarrow X$ the irreducible components of the fixed locus for this action and by $N_{F_k}$ their normal bundles. The natural map:
$$
\begin{array}{ccc}
A^\ast_{\C^\ast}(X) \otimes \mathbb{C}(t) & \rightarrow & \sum_{k}{A^\ast_{\C^\ast}}(F_k) \otimes \mathbb{C}(t)\\
                                             &             &                                        \\
\alpha                                       & \mapsto     &\displaystyle{\frac{i_k^\ast\alpha}{c_{top}(N_{F_k})}}.
\end{array}
$$
is an isomorphism. Pushing forward equivariantly to the class of a point, one has the Atiyah-Bott integration formula:
\begin{equation}\label{eq:abloc}\int_{[X]}\alpha = \sum_k \int_{[F_k]} \frac{i_k^\ast\alpha}{c_{top}(N_{F_k})}.
\end{equation}

Let $\mathbb{C}^\ast$ act on a two-dimensional vector space $V$ via:
$$t\cdot(z_0,z_1)=(tz_0,z_1).$$
This action descends to $\Pro^1$  with fixed points $0=(1:0)$ and $\infty=(0:1)$. An equivariant lift of the $\C^\ast$ action to a line bundle $\Ocal_{\Pro^1}(d)$ over $\Pro^1$ is uniquely determined by the representations (i.e. line bundles over $B\C^\ast$)  $L_{j(0)},L_{j(\infty)}$ of the fibers over the fixed points. One may check that the weights $(j(0), j(\infty))$ satisfy $j(0)-j(\infty) = d$.

\section{The degree of Hurwitz-Hodge classes $\lambda_1^e$}

In this section we compute the degree of the classes  $\lambda_1^e$ on moduli spaces of cyclic covers with exactly four branch points. We use the Atyiah-Bott localization formula \eqref{eq:abloc} to obtain a relation that allows us to determine the desired degrees  in terms of the known degrees of  $\psi$ classes and of zero-dimensional boundary strata.

\label{loc:orb}

 We repeat the statement of  Theorem \ref{thm:orbla} in a slightly different way which, while less compact, may be more transparent.

\begin{theoremola}
 Let $d$ be a positive integer, $0\leq e<d$ and $m_1, m_2,m_3,m_4$ be a monodromy datum.
Assume, without loss of generality, that 
$$
\left\langle\frac{em_1}{d}\right\rangle \leq\left\langle\frac{em_2}{d}\right\rangle  \leq\left\langle\frac{em_3}{d}\right\rangle 
\leq \left\langle\frac{em_4}{d}\right\rangle. 
$$
The degree of the orbifold class $\lambda^e_1$ on the one-dimensional space of degree $d$ cyclic admissible covers of a rational curve with monodromies $m_i$ is given by the following  formula:
\begin{equation}\label{eq:onedimorbb}
\int_{\Adm{m_1, m_2, m_3, m_4}{d}}\lambda^e_1 =
\left\{
\begin{array}{cl}
0 & \mbox{if} \ \sum_{i=1}^4\left\langle\frac{em_i}{d}\right\rangle = 0\\ & \\
0 & \mbox{if} \ \sum_{i=1}^4\left\langle\frac{em_i}{d}\right\rangle = 1\\ & \\
\frac{1}{d}\left\langle\frac{em_1}{d}\right\rangle & \mbox{if}\ \sum_{i=1}^4\left\langle\frac{em_i}{d}\right\rangle = 2\
\mbox{and} \ \left\langle\frac{em_1}{d}\right\rangle +\left\langle\frac{em_4}{d}\right\rangle \leq 1 \\ & \\ 
\frac{1}{d}\left(1-\left\langle\frac{em_4}{d}\right\rangle\right) & \mbox{if}\ \ \sum_{i=1}^4\left\langle\frac{em_i}{d}\right\rangle = 2\
\mbox{and} \ \left\langle\frac{em_1}{d}\right\rangle +\left\langle\frac{em_4}{d}\right\rangle > 1 \\ & \\
0 &\mbox{if} \ \sum_{i=1}^4\left\langle\frac{em_i}{d}\right\rangle = 3.
\end{array}
\right.
\end{equation}
\end{theoremola}

\begin{proof}
We observe  that the class $\lambda_1^0 = 0$, as there are no invariant one-forms on an admissible cover of a rational curve (and therefore $\bE_0$ is a rank $0$ bundle). This proves the first line in \eqref{eq:onedimorbb}.

By the orbifold Riemann-Roch computation \eqref{eq:rankhodgeorb}, the rank of $\bE_e$ is zero when $\sum_{i=1}^4 \left\langle\frac{em_i}{d} \right\rangle=1$, which implies that $\lambda_1^e = 0$ in this case. Since $\lambda_1^e = \lambda_1^{d-e}$ by Lemma \ref{lem:sym} and $\sum_{i=1}^4 \left\langle\frac{em_i}{d} \right\rangle=1$ if an only if $\sum_{i=1}^4 \left\langle\frac{(d-e)m_i}{d} \right\rangle=3$, we obtain that the class vanishes when the sum of the ages is  $3$. Thus the second and fifth lines of \eqref{eq:onedimorbb} are established.

The third and fourth lines are also equivalent: $e,m_1, m_2, m_3, m_4$ satisfy the two numerical conditions of the third line if and only if $d-e, m_1, m_2, m_3, m_4$ satisfy the conditions from the fourth line; the ordering of the fractional parts of $(d-e)m_i/d$ is reversed hence the smallest term is
$$
\left\langle\frac{(d-e)m_4}{d} \right\rangle = 1- \left\langle\frac{em_4}{d} \right\rangle.
$$
Thus establishing that the fourth line in \eqref{eq:onedimorbb} holds completes the proof of the theorem. 
\subsubsection*{The auxiliary integral} Given a monodromy datum $(m_1, m_2,m_3, m_4)$ and an integer $1\leq e \leq d-1$ satisfying the numerical conditions in the fourth line of \eqref{eq:onedimorbb}, we  consider the space of parameterized admissible covers, which we denote by $\Adm{\Pro^1| m_1, m_2,m_3, m_4}{d}$. Letting $f, \pi$ denote the tautological morphisms as in \eqref{eq:acuc}, we have 
\begin{equation} \label{eq:auxint}
    \int_{\Adm{\Pro^1| m_1, m_2,m_3, m_4}{d}} ev_4^\ast(c_1(\Ocal_{\Pro^1}(1))) \cdot c_2(R^1\pi_\ast f^\ast(\Ocal_{\Pro^1}(-1)\boxtimes \Ocal_e)) = 0;
\end{equation}
the integral \eqref{eq:auxint} vanishes for dimension reasons: we are integrating a class of degree $3$ on a space of dimension $4$.

The natural $\C^\ast$ action on $\Pro^1$ induces a torus action on $\Adm{\Pro^1| m_1, m_2,m_3, m_4}{d}$ by post-composition. We may therefore consider equivariant lifts of the integrands and evaluate the integral using the localization formula \eqref{eq:abloc}. Integration in equivariant cohomology  yields a polynomial in the equivariant parameter, hence the fact that the degree of the integrand is stricly less than the dimension of the space implies that the vanishing of \eqref{eq:auxint} continues to hold, regardless of the choice of equivariant lifts of the classes. 

We choose the linearization $(j(0),j(\infty)) = (1,0)$ for the bundle $\Ocal_{\Pro^1}(1)$ so that the first Chern class is represented by the class of the fixed point $0\in \Pro^1$. 
For the bundle $\Ocal_{\Pro^1}(-1)$ we choose the linearization $(0, 1)$.

\subsubsection*{Fixed loci.} An admissible cover is fixed under the $\C^\ast$-action when the four evaluation morphisms have image contained in the fixed locus of $\Pro^1$, i.e. the two points $0, \infty$. Fixed loci may be indexed by elements of the power set $\Pcal([4])$, assigning to a subset $I$ the locus $\Gamma_I$ of maps where the marked points in $I$ are mapped to $\infty$, and those in $I^c$ to $0$.

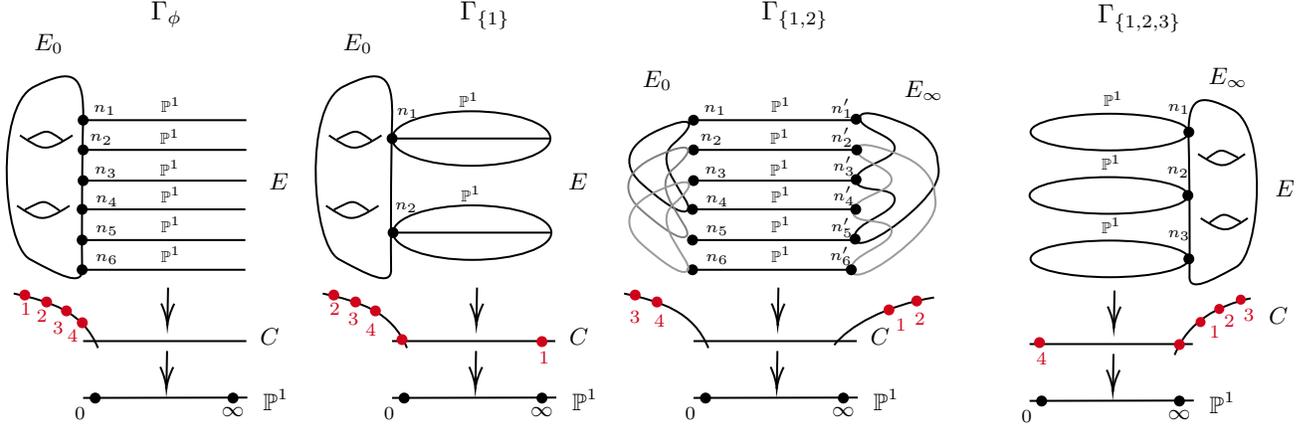
\begin{figure}
    \centering
\tikzset{every picture/.style={line width=0.75pt}} 

\begin{tikzpicture}[x=0.75pt,y=0.75pt,yscale=-1,xscale=1]

\draw    (40.17,280.17) -- (123,280) ;
\draw    (40.17,251.5) -- (122.5,251.5) ;
\draw  [fill={rgb, 255:red, 0; green, 0; blue, 0 }  ,fill opacity=1 ] (44,280.25) .. controls (44,279.01) and (45.01,278) .. (46.25,278) .. controls (47.49,278) and (48.5,279.01) .. (48.5,280.25) .. controls (48.5,281.49) and (47.49,282.5) .. (46.25,282.5) .. controls (45.01,282.5) and (44,281.49) .. (44,280.25) -- cycle ;
\draw  [fill={rgb, 255:red, 0; green, 0; blue, 0 }  ,fill opacity=1 ] (113.5,280.25) .. controls (113.5,279.01) and (114.51,278) .. (115.75,278) .. controls (116.99,278) and (118,279.01) .. (118,280.25) .. controls (118,281.49) and (116.99,282.5) .. (115.75,282.5) .. controls (114.51,282.5) and (113.5,281.49) .. (113.5,280.25) -- cycle ;
\draw    (5.17,227.5) .. controls (23.67,231.5) and (40.17,236.5) .. (47.67,255) ;
\draw    (82.5,256.5) -- (82.2,273.5) ;
\draw [shift={(82.17,275.5)}, rotate = 271.01] [color={rgb, 255:red, 0; green, 0; blue, 0 }  ][line width=0.75]    (10.93,-3.29) .. controls (6.95,-1.4) and (3.31,-0.3) .. (0,0) .. controls (3.31,0.3) and (6.95,1.4) .. (10.93,3.29)   ;
\draw    (83,224.5) -- (82.7,241.5) ;
\draw [shift={(82.67,243.5)}, rotate = 271.01] [color={rgb, 255:red, 0; green, 0; blue, 0 }  ][line width=0.75]    (10.93,-3.29) .. controls (6.95,-1.4) and (3.31,-0.3) .. (0,0) .. controls (3.31,0.3) and (6.95,1.4) .. (10.93,3.29)   ;
\draw  [color={rgb, 255:red, 208; green, 2; blue, 27 }  ,draw opacity=1 ][fill={rgb, 255:red, 208; green, 2; blue, 27 }  ,fill opacity=1 ] (8.5,228.25) .. controls (8.5,227.01) and (9.51,226) .. (10.75,226) .. controls (11.99,226) and (13,227.01) .. (13,228.25) .. controls (13,229.49) and (11.99,230.5) .. (10.75,230.5) .. controls (9.51,230.5) and (8.5,229.49) .. (8.5,228.25) -- cycle ;
\draw  [color={rgb, 255:red, 208; green, 2; blue, 27 }  ,draw opacity=1 ][fill={rgb, 255:red, 208; green, 2; blue, 27 }  ,fill opacity=1 ] (19.5,231.75) .. controls (19.5,230.51) and (20.51,229.5) .. (21.75,229.5) .. controls (22.99,229.5) and (24,230.51) .. (24,231.75) .. controls (24,232.99) and (22.99,234) .. (21.75,234) .. controls (20.51,234) and (19.5,232.99) .. (19.5,231.75) -- cycle ;
\draw  [color={rgb, 255:red, 208; green, 2; blue, 27 }  ,draw opacity=1 ][fill={rgb, 255:red, 208; green, 2; blue, 27 }  ,fill opacity=1 ] (29.5,236.25) .. controls (29.5,235.01) and (30.51,234) .. (31.75,234) .. controls (32.99,234) and (34,235.01) .. (34,236.25) .. controls (34,237.49) and (32.99,238.5) .. (31.75,238.5) .. controls (30.51,238.5) and (29.5,237.49) .. (29.5,236.25) -- cycle ;
\draw  [color={rgb, 255:red, 208; green, 2; blue, 27 }  ,draw opacity=1 ][fill={rgb, 255:red, 208; green, 2; blue, 27 }  ,fill opacity=1 ] (37.5,242.25) .. controls (37.5,241.01) and (38.51,240) .. (39.75,240) .. controls (40.99,240) and (42,241.01) .. (42,242.25) .. controls (42,243.49) and (40.99,244.5) .. (39.75,244.5) .. controls (38.51,244.5) and (37.5,243.49) .. (37.5,242.25) -- cycle ;
\draw    (39.67,215.5) -- (122,215.5) ;
\draw    (39.67,200.5) -- (122,200.5) ;
\draw    (39.67,185) -- (122,185) ;
\draw    (40.17,170.5) -- (122.5,170.5) ;
\draw    (40.17,155) -- (122.5,155) ;
\draw    (40.17,140) -- (122.5,140) ;
\draw   (15.33,123.17) .. controls (35.33,113.17) and (40.83,116.5) .. (40.17,140) .. controls (39.5,163.5) and (39.5,206.5) .. (39.67,215.5) .. controls (39.83,224.5) and (27.17,217.67) .. (13.33,210.17) .. controls (-0.5,202.67) and (-4.67,133.17) .. (15.33,123.17) -- cycle ;
\draw    (8.33,147.67) .. controls (19.83,155.17) and (20.83,156.67) .. (34.83,147.67) ;
\draw    (12.33,149.67) .. controls (22.33,144.67) and (21.33,144.17) .. (29.83,150.17) ;
\draw    (6.83,183.17) .. controls (18.33,190.67) and (19.33,192.17) .. (33.33,183.17) ;
\draw    (10.83,185.17) .. controls (20.83,180.17) and (19.83,179.67) .. (28.33,185.67) ;
\draw  [fill={rgb, 255:red, 0; green, 0; blue, 0 }  ,fill opacity=1 ] (37.42,215.5) .. controls (37.42,214.26) and (38.42,213.25) .. (39.67,213.25) .. controls (40.91,213.25) and (41.92,214.26) .. (41.92,215.5) .. controls (41.92,216.74) and (40.91,217.75) .. (39.67,217.75) .. controls (38.42,217.75) and (37.42,216.74) .. (37.42,215.5) -- cycle ;
\draw  [fill={rgb, 255:red, 0; green, 0; blue, 0 }  ,fill opacity=1 ] (37.42,200.5) .. controls (37.42,199.26) and (38.42,198.25) .. (39.67,198.25) .. controls (40.91,198.25) and (41.92,199.26) .. (41.92,200.5) .. controls (41.92,201.74) and (40.91,202.75) .. (39.67,202.75) .. controls (38.42,202.75) and (37.42,201.74) .. (37.42,200.5) -- cycle ;
\draw  [fill={rgb, 255:red, 0; green, 0; blue, 0 }  ,fill opacity=1 ] (37.42,185) .. controls (37.42,183.76) and (38.42,182.75) .. (39.67,182.75) .. controls (40.91,182.75) and (41.92,183.76) .. (41.92,185) .. controls (41.92,186.24) and (40.91,187.25) .. (39.67,187.25) .. controls (38.42,187.25) and (37.42,186.24) .. (37.42,185) -- cycle ;
\draw  [fill={rgb, 255:red, 0; green, 0; blue, 0 }  ,fill opacity=1 ] (37.92,170.5) .. controls (37.92,169.26) and (38.92,168.25) .. (40.17,168.25) .. controls (41.41,168.25) and (42.42,169.26) .. (42.42,170.5) .. controls (42.42,171.74) and (41.41,172.75) .. (40.17,172.75) .. controls (38.92,172.75) and (37.92,171.74) .. (37.92,170.5) -- cycle ;
\draw  [fill={rgb, 255:red, 0; green, 0; blue, 0 }  ,fill opacity=1 ] (37.92,155) .. controls (37.92,153.76) and (38.92,152.75) .. (40.17,152.75) .. controls (41.41,152.75) and (42.42,153.76) .. (42.42,155) .. controls (42.42,156.24) and (41.41,157.25) .. (40.17,157.25) .. controls (38.92,157.25) and (37.92,156.24) .. (37.92,155) -- cycle ;
\draw  [fill={rgb, 255:red, 0; green, 0; blue, 0 }  ,fill opacity=1 ] (37.92,140) .. controls (37.92,138.76) and (38.92,137.75) .. (40.17,137.75) .. controls (41.41,137.75) and (42.42,138.76) .. (42.42,140) .. controls (42.42,141.24) and (41.41,142.25) .. (40.17,142.25) .. controls (38.92,142.25) and (37.92,141.24) .. (37.92,140) -- cycle ;
\draw    (196,280.17) -- (278.83,280) ;
\draw    (196,251.5) -- (278.33,251.5) ;
\draw  [fill={rgb, 255:red, 0; green, 0; blue, 0 }  ,fill opacity=1 ] (199.83,280.25) .. controls (199.83,279.01) and (200.84,278) .. (202.08,278) .. controls (203.33,278) and (204.33,279.01) .. (204.33,280.25) .. controls (204.33,281.49) and (203.33,282.5) .. (202.08,282.5) .. controls (200.84,282.5) and (199.83,281.49) .. (199.83,280.25) -- cycle ;
\draw  [fill={rgb, 255:red, 0; green, 0; blue, 0 }  ,fill opacity=1 ] (269.33,280.25) .. controls (269.33,279.01) and (270.34,278) .. (271.58,278) .. controls (272.83,278) and (273.83,279.01) .. (273.83,280.25) .. controls (273.83,281.49) and (272.83,282.5) .. (271.58,282.5) .. controls (270.34,282.5) and (269.33,281.49) .. (269.33,280.25) -- cycle ;
\draw    (161,227.5) .. controls (179.5,231.5) and (196,236.5) .. (203.5,255) ;
\draw    (238.33,256.5) -- (238.04,273.5) ;
\draw [shift={(238,275.5)}, rotate = 271.01] [color={rgb, 255:red, 0; green, 0; blue, 0 }  ][line width=0.75]    (10.93,-3.29) .. controls (6.95,-1.4) and (3.31,-0.3) .. (0,0) .. controls (3.31,0.3) and (6.95,1.4) .. (10.93,3.29)   ;
\draw    (238.83,224.5) -- (238.54,241.5) ;
\draw [shift={(238.5,243.5)}, rotate = 271.01] [color={rgb, 255:red, 0; green, 0; blue, 0 }  ][line width=0.75]    (10.93,-3.29) .. controls (6.95,-1.4) and (3.31,-0.3) .. (0,0) .. controls (3.31,0.3) and (6.95,1.4) .. (10.93,3.29)   ;
\draw  [color={rgb, 255:red, 208; green, 2; blue, 27 }  ,draw opacity=1 ][fill={rgb, 255:red, 208; green, 2; blue, 27 }  ,fill opacity=1 ] (164.33,228.25) .. controls (164.33,227.01) and (165.34,226) .. (166.58,226) .. controls (167.83,226) and (168.83,227.01) .. (168.83,228.25) .. controls (168.83,229.49) and (167.83,230.5) .. (166.58,230.5) .. controls (165.34,230.5) and (164.33,229.49) .. (164.33,228.25) -- cycle ;
\draw  [color={rgb, 255:red, 208; green, 2; blue, 27 }  ,draw opacity=1 ][fill={rgb, 255:red, 208; green, 2; blue, 27 }  ,fill opacity=1 ] (175.33,231.75) .. controls (175.33,230.51) and (176.34,229.5) .. (177.58,229.5) .. controls (178.83,229.5) and (179.83,230.51) .. (179.83,231.75) .. controls (179.83,232.99) and (178.83,234) .. (177.58,234) .. controls (176.34,234) and (175.33,232.99) .. (175.33,231.75) -- cycle ;
\draw  [color={rgb, 255:red, 208; green, 2; blue, 27 }  ,draw opacity=1 ][fill={rgb, 255:red, 208; green, 2; blue, 27 }  ,fill opacity=1 ] (185.33,236.25) .. controls (185.33,235.01) and (186.34,234) .. (187.58,234) .. controls (188.83,234) and (189.83,235.01) .. (189.83,236.25) .. controls (189.83,237.49) and (188.83,238.5) .. (187.58,238.5) .. controls (186.34,238.5) and (185.33,237.49) .. (185.33,236.25) -- cycle ;
\draw  [color={rgb, 255:red, 208; green, 2; blue, 27 }  ,draw opacity=1 ][fill={rgb, 255:red, 208; green, 2; blue, 27 }  ,fill opacity=1 ] (269.33,251.75) .. controls (269.33,250.51) and (270.34,249.5) .. (271.58,249.5) .. controls (272.83,249.5) and (273.83,250.51) .. (273.83,251.75) .. controls (273.83,252.99) and (272.83,254) .. (271.58,254) .. controls (270.34,254) and (269.33,252.99) .. (269.33,251.75) -- cycle ;
\draw   (171.17,123.17) .. controls (191.17,113.17) and (196.67,116.5) .. (196,140) .. controls (195.33,163.5) and (195.33,206.5) .. (195.5,215.5) .. controls (195.67,224.5) and (183,217.67) .. (169.17,210.17) .. controls (155.33,202.67) and (151.17,133.17) .. (171.17,123.17) -- cycle ;
\draw    (164.17,147.67) .. controls (175.67,155.17) and (176.67,156.67) .. (190.67,147.67) ;
\draw    (168.17,149.67) .. controls (178.17,144.67) and (177.17,144.17) .. (185.67,150.17) ;
\draw    (162.67,183.17) .. controls (174.17,190.67) and (175.17,192.17) .. (189.17,183.17) ;
\draw    (166.67,185.17) .. controls (176.67,180.17) and (175.67,179.67) .. (184.17,185.67) ;
\draw  [fill={rgb, 255:red, 0; green, 0; blue, 0 }  ,fill opacity=1 ] (194.25,196.75) .. controls (194.25,195.51) and (195.26,194.5) .. (196.5,194.5) .. controls (197.74,194.5) and (198.75,195.51) .. (198.75,196.75) .. controls (198.75,197.99) and (197.74,199) .. (196.5,199) .. controls (195.26,199) and (194.25,197.99) .. (194.25,196.75) -- cycle ;
\draw  [fill={rgb, 255:red, 0; green, 0; blue, 0 }  ,fill opacity=1 ] (193.75,149.25) .. controls (193.75,148.01) and (194.76,147) .. (196,147) .. controls (197.24,147) and (198.25,148.01) .. (198.25,149.25) .. controls (198.25,150.49) and (197.24,151.5) .. (196,151.5) .. controls (194.76,151.5) and (193.75,150.49) .. (193.75,149.25) -- cycle ;
\draw  [color={rgb, 255:red, 208; green, 2; blue, 27 }  ,draw opacity=1 ][fill={rgb, 255:red, 208; green, 2; blue, 27 }  ,fill opacity=1 ] (198.75,250.75) .. controls (198.75,249.51) and (199.76,248.5) .. (201,248.5) .. controls (202.24,248.5) and (203.25,249.51) .. (203.25,250.75) .. controls (203.25,251.99) and (202.24,253) .. (201,253) .. controls (199.76,253) and (198.75,251.99) .. (198.75,250.75) -- cycle ;
\draw   (196,149.25) .. controls (196,141.66) and (213.87,135.5) .. (235.92,135.5) .. controls (257.96,135.5) and (275.83,141.66) .. (275.83,149.25) .. controls (275.83,156.84) and (257.96,163) .. (235.92,163) .. controls (213.87,163) and (196,156.84) .. (196,149.25) -- cycle ;
\draw    (196,149.25) -- (275.83,149.25) ;
\draw   (196.5,196.75) .. controls (196.5,189.16) and (214.37,183) .. (236.42,183) .. controls (258.46,183) and (276.33,189.16) .. (276.33,196.75) .. controls (276.33,204.34) and (258.46,210.5) .. (236.42,210.5) .. controls (214.37,210.5) and (196.5,204.34) .. (196.5,196.75) -- cycle ;
\draw    (196.5,196.75) -- (276.33,196.75) ;
\draw    (348,280.17) -- (430.83,280) ;
\draw    (348,251.5) -- (430.33,251.5) ;
\draw  [fill={rgb, 255:red, 0; green, 0; blue, 0 }  ,fill opacity=1 ] (351.83,280.25) .. controls (351.83,279.01) and (352.84,278) .. (354.08,278) .. controls (355.33,278) and (356.33,279.01) .. (356.33,280.25) .. controls (356.33,281.49) and (355.33,282.5) .. (354.08,282.5) .. controls (352.84,282.5) and (351.83,281.49) .. (351.83,280.25) -- cycle ;
\draw  [fill={rgb, 255:red, 0; green, 0; blue, 0 }  ,fill opacity=1 ] (421.33,280.25) .. controls (421.33,279.01) and (422.34,278) .. (423.58,278) .. controls (424.83,278) and (425.83,279.01) .. (425.83,280.25) .. controls (425.83,281.49) and (424.83,282.5) .. (423.58,282.5) .. controls (422.34,282.5) and (421.33,281.49) .. (421.33,280.25) -- cycle ;
\draw    (313,227.5) .. controls (331.5,231.5) and (348,236.5) .. (355.5,255) ;
\draw    (390.33,256.5) -- (390.04,273.5) ;
\draw [shift={(390,275.5)}, rotate = 271.01] [color={rgb, 255:red, 0; green, 0; blue, 0 }  ][line width=0.75]    (10.93,-3.29) .. controls (6.95,-1.4) and (3.31,-0.3) .. (0,0) .. controls (3.31,0.3) and (6.95,1.4) .. (10.93,3.29)   ;
\draw    (390.83,224.5) -- (390.54,241.5) ;
\draw [shift={(390.5,243.5)}, rotate = 271.01] [color={rgb, 255:red, 0; green, 0; blue, 0 }  ][line width=0.75]    (10.93,-3.29) .. controls (6.95,-1.4) and (3.31,-0.3) .. (0,0) .. controls (3.31,0.3) and (6.95,1.4) .. (10.93,3.29)   ;
\draw  [color={rgb, 255:red, 208; green, 2; blue, 27 }  ,draw opacity=1 ][fill={rgb, 255:red, 208; green, 2; blue, 27 }  ,fill opacity=1 ] (316.33,228.25) .. controls (316.33,227.01) and (317.34,226) .. (318.58,226) .. controls (319.83,226) and (320.83,227.01) .. (320.83,228.25) .. controls (320.83,229.49) and (319.83,230.5) .. (318.58,230.5) .. controls (317.34,230.5) and (316.33,229.49) .. (316.33,228.25) -- cycle ;
\draw  [color={rgb, 255:red, 208; green, 2; blue, 27 }  ,draw opacity=1 ][fill={rgb, 255:red, 208; green, 2; blue, 27 }  ,fill opacity=1 ] (327.33,231.75) .. controls (327.33,230.51) and (328.34,229.5) .. (329.58,229.5) .. controls (330.83,229.5) and (331.83,230.51) .. (331.83,231.75) .. controls (331.83,232.99) and (330.83,234) .. (329.58,234) .. controls (328.34,234) and (327.33,232.99) .. (327.33,231.75) -- cycle ;
\draw    (347.5,215.5) -- (429.83,215.5) ;
\draw    (347.5,200.5) -- (429.83,200.5) ;
\draw    (347.5,185) -- (429.83,185) ;
\draw    (348,170.5) -- (430.33,170.5) ;
\draw    (348,155) -- (430.33,155) ;
\draw    (348,140) -- (430.33,140) ;
\draw  [fill={rgb, 255:red, 0; green, 0; blue, 0 }  ,fill opacity=1 ] (345.25,215.5) .. controls (345.25,214.26) and (346.26,213.25) .. (347.5,213.25) .. controls (348.74,213.25) and (349.75,214.26) .. (349.75,215.5) .. controls (349.75,216.74) and (348.74,217.75) .. (347.5,217.75) .. controls (346.26,217.75) and (345.25,216.74) .. (345.25,215.5) -- cycle ;
\draw  [fill={rgb, 255:red, 0; green, 0; blue, 0 }  ,fill opacity=1 ] (345.25,200.5) .. controls (345.25,199.26) and (346.26,198.25) .. (347.5,198.25) .. controls (348.74,198.25) and (349.75,199.26) .. (349.75,200.5) .. controls (349.75,201.74) and (348.74,202.75) .. (347.5,202.75) .. controls (346.26,202.75) and (345.25,201.74) .. (345.25,200.5) -- cycle ;
\draw  [fill={rgb, 255:red, 0; green, 0; blue, 0 }  ,fill opacity=1 ] (345.25,185) .. controls (345.25,183.76) and (346.26,182.75) .. (347.5,182.75) .. controls (348.74,182.75) and (349.75,183.76) .. (349.75,185) .. controls (349.75,186.24) and (348.74,187.25) .. (347.5,187.25) .. controls (346.26,187.25) and (345.25,186.24) .. (345.25,185) -- cycle ;
\draw  [fill={rgb, 255:red, 0; green, 0; blue, 0 }  ,fill opacity=1 ] (345.75,170.5) .. controls (345.75,169.26) and (346.76,168.25) .. (348,168.25) .. controls (349.24,168.25) and (350.25,169.26) .. (350.25,170.5) .. controls (350.25,171.74) and (349.24,172.75) .. (348,172.75) .. controls (346.76,172.75) and (345.75,171.74) .. (345.75,170.5) -- cycle ;
\draw  [fill={rgb, 255:red, 0; green, 0; blue, 0 }  ,fill opacity=1 ] (345.75,155) .. controls (345.75,153.76) and (346.76,152.75) .. (348,152.75) .. controls (349.24,152.75) and (350.25,153.76) .. (350.25,155) .. controls (350.25,156.24) and (349.24,157.25) .. (348,157.25) .. controls (346.76,157.25) and (345.75,156.24) .. (345.75,155) -- cycle ;
\draw  [fill={rgb, 255:red, 0; green, 0; blue, 0 }  ,fill opacity=1 ] (345.75,140) .. controls (345.75,138.76) and (346.76,137.75) .. (348,137.75) .. controls (349.24,137.75) and (350.25,138.76) .. (350.25,140) .. controls (350.25,141.24) and (349.24,142.25) .. (348,142.25) .. controls (346.76,142.25) and (345.75,141.24) .. (345.75,140) -- cycle ;
\draw    (419,253.5) .. controls (426.33,244) and (444.83,233.5) .. (469.33,229.5) ;
\draw  [color={rgb, 255:red, 208; green, 2; blue, 27 }  ,draw opacity=1 ][fill={rgb, 255:red, 208; green, 2; blue, 27 }  ,fill opacity=1 ] (458.33,231.25) .. controls (458.33,230.01) and (459.34,229) .. (460.58,229) .. controls (461.83,229) and (462.83,230.01) .. (462.83,231.25) .. controls (462.83,232.49) and (461.83,233.5) .. (460.58,233.5) .. controls (459.34,233.5) and (458.33,232.49) .. (458.33,231.25) -- cycle ;
\draw  [color={rgb, 255:red, 208; green, 2; blue, 27 }  ,draw opacity=1 ][fill={rgb, 255:red, 208; green, 2; blue, 27 }  ,fill opacity=1 ] (444.33,235.75) .. controls (444.33,234.51) and (445.34,233.5) .. (446.58,233.5) .. controls (447.83,233.5) and (448.83,234.51) .. (448.83,235.75) .. controls (448.83,236.99) and (447.83,238) .. (446.58,238) .. controls (445.34,238) and (444.33,236.99) .. (444.33,235.75) -- cycle ;
\draw  [fill={rgb, 255:red, 0; green, 0; blue, 0 }  ,fill opacity=1 ] (427.25,200) .. controls (427.25,198.76) and (428.26,197.75) .. (429.5,197.75) .. controls (430.74,197.75) and (431.75,198.76) .. (431.75,200) .. controls (431.75,201.24) and (430.74,202.25) .. (429.5,202.25) .. controls (428.26,202.25) and (427.25,201.24) .. (427.25,200) -- cycle ;
\draw  [fill={rgb, 255:red, 0; green, 0; blue, 0 }  ,fill opacity=1 ] (427.75,170) .. controls (427.75,168.76) and (428.76,167.75) .. (430,167.75) .. controls (431.24,167.75) and (432.25,168.76) .. (432.25,170) .. controls (432.25,171.24) and (431.24,172.25) .. (430,172.25) .. controls (428.76,172.25) and (427.75,171.24) .. (427.75,170) -- cycle ;
\draw  [fill={rgb, 255:red, 0; green, 0; blue, 0 }  ,fill opacity=1 ] (427.75,139.5) .. controls (427.75,138.26) and (428.76,137.25) .. (430,137.25) .. controls (431.24,137.25) and (432.25,138.26) .. (432.25,139.5) .. controls (432.25,140.74) and (431.24,141.75) .. (430,141.75) .. controls (428.76,141.75) and (427.75,140.74) .. (427.75,139.5) -- cycle ;
\draw   (330.33,145) .. controls (349.33,136.5) and (349.5,142) .. (337.83,152) .. controls (326.17,162) and (347.67,180.5) .. (345.25,185) .. controls (342.83,189.5) and (330.33,179.5) .. (321.83,172) .. controls (313.33,164.5) and (311.33,153.5) .. (330.33,145) -- cycle ;
\draw  [color={rgb, 255:red, 128; green, 128; blue, 128 }  ,draw opacity=1 ] (329.83,158) .. controls (348.83,149.5) and (349,155) .. (337.33,165) .. controls (325.67,175) and (347.17,193.5) .. (344.75,198) .. controls (342.33,202.5) and (329.83,192.5) .. (321.33,185) .. controls (312.83,177.5) and (310.83,166.5) .. (329.83,158) -- cycle ;
\draw  [color={rgb, 255:red, 155; green, 155; blue, 155 }  ,draw opacity=1 ] (330.33,175.5) .. controls (349.33,167) and (349.5,172.5) .. (337.83,182.5) .. controls (326.17,192.5) and (347.67,211) .. (345.25,215.5) .. controls (342.83,220) and (330.33,210) .. (321.83,202.5) .. controls (313.33,195) and (311.33,184) .. (330.33,175.5) -- cycle ;
\draw   (471,159.5) .. controls (470.62,154.2) and (465.67,147) .. (457,143) .. controls (448.33,139) and (431.59,133.07) .. (431.92,139) .. controls (432.25,144.93) and (449.5,147.5) .. (449.5,154) .. controls (449.5,160.5) and (430.5,162.5) .. (430,170) .. controls (429.5,177.5) and (449.5,172.5) .. (449.5,180) .. controls (449.5,187.5) and (428.5,191.5) .. (429.17,199.5) .. controls (429.83,207.5) and (450.35,194.59) .. (462.5,182) .. controls (474.65,169.41) and (471.38,164.8) .. (471,159.5) -- cycle ;
\draw  [color={rgb, 255:red, 155; green, 155; blue, 155 }  ,draw opacity=1 ] (469.42,175.5) .. controls (469.04,170.2) and (464.08,163) .. (455.42,159) .. controls (446.75,155) and (430,149.07) .. (430.33,155) .. controls (430.66,160.93) and (447.92,163.5) .. (447.92,170) .. controls (447.92,176.5) and (428.92,178.5) .. (428.42,186) .. controls (427.92,193.5) and (447.92,188.5) .. (447.92,196) .. controls (447.92,203.5) and (426.92,207.5) .. (427.58,215.5) .. controls (428.25,223.5) and (448.77,210.59) .. (460.92,198) .. controls (473.06,185.41) and (469.8,180.8) .. (469.42,175.5) -- cycle ;
\draw  [fill={rgb, 255:red, 0; green, 0; blue, 0 }  ,fill opacity=1 ] (428.08,155) .. controls (428.08,153.76) and (429.09,152.75) .. (430.33,152.75) .. controls (431.58,152.75) and (432.58,153.76) .. (432.58,155) .. controls (432.58,156.24) and (431.58,157.25) .. (430.33,157.25) .. controls (429.09,157.25) and (428.08,156.24) .. (428.08,155) -- cycle ;
\draw  [fill={rgb, 255:red, 0; green, 0; blue, 0 }  ,fill opacity=1 ] (427.58,185) .. controls (427.58,183.76) and (428.59,182.75) .. (429.83,182.75) .. controls (431.08,182.75) and (432.08,183.76) .. (432.08,185) .. controls (432.08,186.24) and (431.08,187.25) .. (429.83,187.25) .. controls (428.59,187.25) and (427.58,186.24) .. (427.58,185) -- cycle ;
\draw  [fill={rgb, 255:red, 0; green, 0; blue, 0 }  ,fill opacity=1 ] (425.33,215.5) .. controls (425.33,214.26) and (426.34,213.25) .. (427.58,213.25) .. controls (428.83,213.25) and (429.83,214.26) .. (429.83,215.5) .. controls (429.83,216.74) and (428.83,217.75) .. (427.58,217.75) .. controls (426.34,217.75) and (425.33,216.74) .. (425.33,215.5) -- cycle ;
\draw    (517.5,281.67) -- (600.33,281.5) ;
\draw    (517.5,253) -- (599.83,253) ;
\draw  [fill={rgb, 255:red, 0; green, 0; blue, 0 }  ,fill opacity=1 ] (521.33,281.75) .. controls (521.33,280.51) and (522.34,279.5) .. (523.58,279.5) .. controls (524.83,279.5) and (525.83,280.51) .. (525.83,281.75) .. controls (525.83,282.99) and (524.83,284) .. (523.58,284) .. controls (522.34,284) and (521.33,282.99) .. (521.33,281.75) -- cycle ;
\draw  [fill={rgb, 255:red, 0; green, 0; blue, 0 }  ,fill opacity=1 ] (590.83,281.75) .. controls (590.83,280.51) and (591.84,279.5) .. (593.08,279.5) .. controls (594.33,279.5) and (595.33,280.51) .. (595.33,281.75) .. controls (595.33,282.99) and (594.33,284) .. (593.08,284) .. controls (591.84,284) and (590.83,282.99) .. (590.83,281.75) -- cycle ;
\draw    (629.01,229.91) .. controls (612.46,233.56) and (597.37,241.62) .. (590.67,258.5) ;
\draw    (559.83,258) -- (559.54,275) ;
\draw [shift={(559.5,277)}, rotate = 271.01] [color={rgb, 255:red, 0; green, 0; blue, 0 }  ][line width=0.75]    (10.93,-3.29) .. controls (6.95,-1.4) and (3.31,-0.3) .. (0,0) .. controls (3.31,0.3) and (6.95,1.4) .. (10.93,3.29)   ;
\draw    (560.33,226) -- (560.04,243) ;
\draw [shift={(560,245)}, rotate = 271.01] [color={rgb, 255:red, 0; green, 0; blue, 0 }  ][line width=0.75]    (10.93,-3.29) .. controls (6.95,-1.4) and (3.31,-0.3) .. (0,0) .. controls (3.31,0.3) and (6.95,1.4) .. (10.93,3.29)   ;
\draw  [color={rgb, 255:red, 208; green, 2; blue, 27 }  ,draw opacity=1 ][fill={rgb, 255:red, 208; green, 2; blue, 27 }  ,fill opacity=1 ] (626.03,230.59) .. controls (626.03,229.46) and (625.12,228.54) .. (624.01,228.54) .. controls (622.9,228.54) and (622,229.46) .. (622,230.59) .. controls (622,231.72) and (622.9,232.64) .. (624.01,232.64) .. controls (625.12,232.64) and (626.03,231.72) .. (626.03,230.59) -- cycle ;
\draw  [color={rgb, 255:red, 208; green, 2; blue, 27 }  ,draw opacity=1 ][fill={rgb, 255:red, 208; green, 2; blue, 27 }  ,fill opacity=1 ] (615.19,235.28) .. controls (615.19,234.15) and (614.29,233.23) .. (613.18,233.23) .. controls (612.07,233.23) and (611.16,234.15) .. (611.16,235.28) .. controls (611.16,236.42) and (612.07,237.34) .. (613.18,237.34) .. controls (614.29,237.34) and (615.19,236.42) .. (615.19,235.28) -- cycle ;
\draw  [color={rgb, 255:red, 208; green, 2; blue, 27 }  ,draw opacity=1 ][fill={rgb, 255:red, 208; green, 2; blue, 27 }  ,fill opacity=1 ] (605.75,241.89) .. controls (605.75,240.76) and (604.85,239.84) .. (603.73,239.84) .. controls (602.62,239.84) and (601.72,240.76) .. (601.72,241.89) .. controls (601.72,243.02) and (602.62,243.94) .. (603.73,243.94) .. controls (604.85,243.94) and (605.75,243.02) .. (605.75,241.89) -- cycle ;
\draw  [color={rgb, 255:red, 208; green, 2; blue, 27 }  ,draw opacity=1 ][fill={rgb, 255:red, 208; green, 2; blue, 27 }  ,fill opacity=1 ] (590.83,253.25) .. controls (590.83,252.01) and (591.84,251) .. (593.08,251) .. controls (594.33,251) and (595.33,252.01) .. (595.33,253.25) .. controls (595.33,254.49) and (594.33,255.5) .. (593.08,255.5) .. controls (591.84,255.5) and (590.83,254.49) .. (590.83,253.25) -- cycle ;
\draw   (619.91,134.7) .. controls (602.03,125.58) and (597.11,128.62) .. (597.71,150.06) .. controls (598.3,171.51) and (598.3,210.74) .. (598.15,218.96) .. controls (598.01,227.17) and (609.33,220.93) .. (621.7,214.09) .. controls (634.07,207.25) and (637.8,143.83) .. (619.91,134.7) -- cycle ;
\draw    (626.17,157.06) .. controls (615.89,163.9) and (615,165.27) .. (602.48,157.06) ;
\draw    (622.6,158.88) .. controls (613.65,154.32) and (614.55,153.86) .. (606.95,159.34) ;
\draw    (627.52,189.45) .. controls (617.23,196.3) and (616.34,197.66) .. (603.82,189.45) ;
\draw    (623.94,191.28) .. controls (615,186.71) and (615.89,186.26) .. (608.29,191.73) ;
\draw  [fill={rgb, 255:red, 0; green, 0; blue, 0 }  ,fill opacity=1 ] (595.08,178) .. controls (595.08,176.76) and (596.09,175.75) .. (597.33,175.75) .. controls (598.58,175.75) and (599.58,176.76) .. (599.58,178) .. controls (599.58,179.24) and (598.58,180.25) .. (597.33,180.25) .. controls (596.09,180.25) and (595.08,179.24) .. (595.08,178) -- cycle ;
\draw  [fill={rgb, 255:red, 0; green, 0; blue, 0 }  ,fill opacity=1 ] (595.58,146) .. controls (595.58,144.76) and (596.59,143.75) .. (597.83,143.75) .. controls (599.08,143.75) and (600.08,144.76) .. (600.08,146) .. controls (600.08,147.24) and (599.08,148.25) .. (597.83,148.25) .. controls (596.59,148.25) and (595.58,147.24) .. (595.58,146) -- cycle ;
\draw  [color={rgb, 255:red, 208; green, 2; blue, 27 }  ,draw opacity=1 ][fill={rgb, 255:red, 208; green, 2; blue, 27 }  ,fill opacity=1 ] (520.25,252.25) .. controls (520.25,251.01) and (521.26,250) .. (522.5,250) .. controls (523.74,250) and (524.75,251.01) .. (524.75,252.25) .. controls (524.75,253.49) and (523.74,254.5) .. (522.5,254.5) .. controls (521.26,254.5) and (520.25,253.49) .. (520.25,252.25) -- cycle ;
\draw   (518,146) .. controls (518,140.89) and (535.87,136.75) .. (557.92,136.75) .. controls (579.96,136.75) and (597.83,140.89) .. (597.83,146) .. controls (597.83,151.11) and (579.96,155.25) .. (557.92,155.25) .. controls (535.87,155.25) and (518,151.11) .. (518,146) -- cycle ;
\draw   (517.5,178) .. controls (517.5,172.89) and (535.37,168.75) .. (557.42,168.75) .. controls (579.46,168.75) and (597.33,172.89) .. (597.33,178) .. controls (597.33,183.11) and (579.46,187.25) .. (557.42,187.25) .. controls (535.37,187.25) and (517.5,183.11) .. (517.5,178) -- cycle ;
\draw   (518,210) .. controls (518,204.89) and (535.87,200.75) .. (557.92,200.75) .. controls (579.96,200.75) and (597.83,204.89) .. (597.83,210) .. controls (597.83,215.11) and (579.96,219.25) .. (557.92,219.25) .. controls (535.87,219.25) and (518,215.11) .. (518,210) -- cycle ;
\draw  [fill={rgb, 255:red, 0; green, 0; blue, 0 }  ,fill opacity=1 ] (595.58,210) .. controls (595.58,208.76) and (596.59,207.75) .. (597.83,207.75) .. controls (599.08,207.75) and (600.08,208.76) .. (600.08,210) .. controls (600.08,211.24) and (599.08,212.25) .. (597.83,212.25) .. controls (596.59,212.25) and (595.58,211.24) .. (595.58,210) -- cycle ;

\draw (34.5,283.4) node [anchor=north west][inner sep=0.75pt]  [font=\scriptsize]  {$0$};
\draw (108.5,283.5) node [anchor=north west][inner sep=0.75pt]  [font=\footnotesize] [align=left] {$\displaystyle \infty $};
\draw (129.5,274.4) node [anchor=north west][inner sep=0.75pt]  [font=\small]  {$\mathbb{P}^{1}$};
\draw (128,243.4) node [anchor=north west][inner sep=0.75pt]  [font=\small]  {$C$};
\draw (7.17,230.5) node [anchor=north west][inner sep=0.75pt]  [font=\footnotesize] [align=left] {{\scriptsize \textcolor[rgb]{0.82,0.01,0.11}{1}}};
\draw (15.25,234) node [anchor=north west][inner sep=0.75pt]  [font=\footnotesize] [align=left] {{\scriptsize \textcolor[rgb]{0.82,0.01,0.11}{2}}};
\draw (23.5,240) node [anchor=north west][inner sep=0.75pt]  [font=\footnotesize] [align=left] {{\scriptsize \textcolor[rgb]{0.82,0.01,0.11}{3}}};
\draw (30.83,244.17) node [anchor=north west][inner sep=0.75pt]  [font=\footnotesize] [align=left] {{\scriptsize \textcolor[rgb]{0.82,0.01,0.11}{4}}};
\draw (132.5,165.07) node [anchor=north west][inner sep=0.75pt]  [font=\small]  {$E$};
\draw (14,95.07) node [anchor=north west][inner sep=0.75pt]  [font=\small]  {$E_{0}$};
\draw (78,126.9) node [anchor=north west][inner sep=0.75pt]  [font=\tiny]  {$\mathbb{P}^{1}$};
\draw (77.5,142.9) node [anchor=north west][inner sep=0.75pt]  [font=\tiny]  {$\mathbb{P}^{1}$};
\draw (77.5,158.9) node [anchor=north west][inner sep=0.75pt]  [font=\tiny]  {$\mathbb{P}^{1}$};
\draw (78,172.4) node [anchor=north west][inner sep=0.75pt]  [font=\tiny]  {$\mathbb{P}^{1}$};
\draw (78,187.4) node [anchor=north west][inner sep=0.75pt]  [font=\tiny]  {$\mathbb{P}^{1}$};
\draw (78,202.4) node [anchor=north west][inner sep=0.75pt]  [font=\tiny]  {$\mathbb{P}^{1}$};
\draw (44,130.73) node [anchor=north west][inner sep=0.75pt]  [font=\tiny]  {$n_{1}$};
\draw (42.17,145.65) node [anchor=north west][inner sep=0.75pt]  [font=\tiny]  {$n_{2}$};
\draw (44.5,162.23) node [anchor=north west][inner sep=0.75pt]  [font=\tiny]  {$n_{3}$};
\draw (45,176.73) node [anchor=north west][inner sep=0.75pt]  [font=\tiny]  {$n_{4}$};
\draw (45,191.73) node [anchor=north west][inner sep=0.75pt]  [font=\tiny]  {$n_{5}$};
\draw (45,206.23) node [anchor=north west][inner sep=0.75pt]  [font=\tiny]  {$n_{6}$};
\draw (73,78.4) node [anchor=north west][inner sep=0.75pt]    {$\Gamma _{\phi }$};
\draw (190.33,283.4) node [anchor=north west][inner sep=0.75pt]  [font=\scriptsize]  {$0$};
\draw (264.33,283.5) node [anchor=north west][inner sep=0.75pt]  [font=\footnotesize] [align=left] {$\displaystyle \infty $};
\draw (285.33,274.4) node [anchor=north west][inner sep=0.75pt]  [font=\small]  {$\mathbb{P}^{1}$};
\draw (283.83,243.4) node [anchor=north west][inner sep=0.75pt]  [font=\small]  {$C$};
\draw (268.5,255.5) node [anchor=north west][inner sep=0.75pt]  [font=\footnotesize] [align=left] {{\scriptsize \textcolor[rgb]{0.82,0.01,0.11}{1}}};
\draw (163,230.5) node [anchor=north west][inner sep=0.75pt]  [font=\footnotesize] [align=left] {{\scriptsize \textcolor[rgb]{0.82,0.01,0.11}{2}}};
\draw (172.33,235.5) node [anchor=north west][inner sep=0.75pt]  [font=\footnotesize] [align=left] {{\scriptsize \textcolor[rgb]{0.82,0.01,0.11}{3}}};
\draw (181.17,241.17) node [anchor=north west][inner sep=0.75pt]  [font=\footnotesize] [align=left] {{\scriptsize \textcolor[rgb]{0.82,0.01,0.11}{4}}};
\draw (283.33,165.07) node [anchor=north west][inner sep=0.75pt]  [font=\small]  {$E$};
\draw (169.83,95.07) node [anchor=north west][inner sep=0.75pt]  [font=\small]  {$E_{0}$};
\draw (228.33,124.4) node [anchor=north west][inner sep=0.75pt]  [font=\tiny]  {$\mathbb{P}^{1}$};
\draw (229.33,171.9) node [anchor=north west][inner sep=0.75pt]  [font=\tiny]  {$\mathbb{P}^{1}$};
\draw (196.83,132.23) node [anchor=north west][inner sep=0.75pt]  [font=\tiny]  {$n_{1}$};
\draw (196,180.15) node [anchor=north west][inner sep=0.75pt]  [font=\tiny]  {$n_{2}$};
\draw (228.83,78.4) node [anchor=north west][inner sep=0.75pt]    {$\Gamma _{\{1\}}$};
\draw (342.33,283.4) node [anchor=north west][inner sep=0.75pt]  [font=\scriptsize]  {$0$};
\draw (416.33,283.5) node [anchor=north west][inner sep=0.75pt]  [font=\footnotesize] [align=left] {$\displaystyle \infty $};
\draw (437.33,274.4) node [anchor=north west][inner sep=0.75pt]  [font=\small]  {$\mathbb{P}^{1}$};
\draw (435.83,243.4) node [anchor=north west][inner sep=0.75pt]  [font=\small]  {$C$};
\draw (448.58,238.75) node [anchor=north west][inner sep=0.75pt]  [font=\footnotesize] [align=left] {{\scriptsize \textcolor[rgb]{0.82,0.01,0.11}{1}}};
\draw (458.08,235.5) node [anchor=north west][inner sep=0.75pt]  [font=\footnotesize] [align=left] {{\scriptsize \textcolor[rgb]{0.82,0.01,0.11}{2}}};
\draw (314.33,234.5) node [anchor=north west][inner sep=0.75pt]  [font=\footnotesize] [align=left] {{\scriptsize \textcolor[rgb]{0.82,0.01,0.11}{3}}};
\draw (326.67,236.67) node [anchor=north west][inner sep=0.75pt]  [font=\footnotesize] [align=left] {{\scriptsize \textcolor[rgb]{0.82,0.01,0.11}{4}}};
\draw (452.83,118.57) node [anchor=north west][inner sep=0.75pt]  [font=\small]  {$E_{\infty }$};
\draw (320.83,113.57) node [anchor=north west][inner sep=0.75pt]  [font=\small]  {$E_{0}$};
\draw (385.83,126.9) node [anchor=north west][inner sep=0.75pt]  [font=\tiny]  {$\mathbb{P}^{1}$};
\draw (385.33,142.9) node [anchor=north west][inner sep=0.75pt]  [font=\tiny]  {$\mathbb{P}^{1}$};
\draw (385.33,158.9) node [anchor=north west][inner sep=0.75pt]  [font=\tiny]  {$\mathbb{P}^{1}$};
\draw (385.83,172.4) node [anchor=north west][inner sep=0.75pt]  [font=\tiny]  {$\mathbb{P}^{1}$};
\draw (385.83,187.4) node [anchor=north west][inner sep=0.75pt]  [font=\tiny]  {$\mathbb{P}^{1}$};
\draw (385.83,202.4) node [anchor=north west][inner sep=0.75pt]  [font=\tiny]  {$\mathbb{P}^{1}$};
\draw (351.83,130.73) node [anchor=north west][inner sep=0.75pt]  [font=\tiny]  {$n_{1}$};
\draw (350,145.65) node [anchor=north west][inner sep=0.75pt]  [font=\tiny]  {$n_{2}$};
\draw (352.33,162.23) node [anchor=north west][inner sep=0.75pt]  [font=\tiny]  {$n_{3}$};
\draw (352.83,176.73) node [anchor=north west][inner sep=0.75pt]  [font=\tiny]  {$n_{4}$};
\draw (352.83,191.73) node [anchor=north west][inner sep=0.75pt]  [font=\tiny]  {$n_{5}$};
\draw (352.83,206.23) node [anchor=north west][inner sep=0.75pt]  [font=\tiny]  {$n_{6}$};
\draw (380.83,78.4) node [anchor=north west][inner sep=0.75pt]    {$\Gamma _{\{1,2\}}$};
\draw (415.83,127.73) node [anchor=north west][inner sep=0.75pt]  [font=\tiny]  {$n_{1} '$};
\draw (415.83,142.23) node [anchor=north west][inner sep=0.75pt]  [font=\tiny]  {$n_{2} '$};
\draw (416.83,156.73) node [anchor=north west][inner sep=0.75pt]  [font=\tiny]  {$n_{3} '$};
\draw (416.83,171.73) node [anchor=north west][inner sep=0.75pt]  [font=\tiny]  {$n_{4} '$};
\draw (415.83,187.73) node [anchor=north west][inner sep=0.75pt]  [font=\tiny]  {$n_{5} '$};
\draw (415.33,202.23) node [anchor=north west][inner sep=0.75pt]  [font=\tiny]  {$n_{6} '$};
\draw (511.83,284.9) node [anchor=north west][inner sep=0.75pt]  [font=\scriptsize]  {$0$};
\draw (585.83,285) node [anchor=north west][inner sep=0.75pt]  [font=\footnotesize] [align=left] {$\displaystyle \infty $};
\draw (606.83,275.9) node [anchor=north west][inner sep=0.75pt]  [font=\small]  {$\mathbb{P}^{1}$};
\draw (636.73,232.84) node [anchor=north west][inner sep=0.75pt]  [font=\small]  {$C$};
\draw (606.73,242.34) node [anchor=north west][inner sep=0.75pt]  [font=\footnotesize] [align=left] {{\scriptsize \textcolor[rgb]{0.82,0.01,0.11}{1}}};
\draw (614.16,236.78) node [anchor=north west][inner sep=0.75pt]  [font=\footnotesize] [align=left] {{\scriptsize \textcolor[rgb]{0.82,0.01,0.11}{2}}};
\draw (624,233.59) node [anchor=north west][inner sep=0.75pt]  [font=\footnotesize] [align=left] {{\scriptsize \textcolor[rgb]{0.82,0.01,0.11}{3}}};
\draw (519.5,256) node [anchor=north west][inner sep=0.75pt]  [font=\footnotesize] [align=left] {{\scriptsize \textcolor[rgb]{0.82,0.01,0.11}{4}}};
\draw (639.83,168.57) node [anchor=north west][inner sep=0.75pt]  [font=\small]  {$E$};
\draw (606.33,112.07) node [anchor=north west][inner sep=0.75pt]  [font=\small]  {$E_{\infty }$};
\draw (553.33,123.9) node [anchor=north west][inner sep=0.75pt]  [font=\tiny]  {$\mathbb{P}^{1}$};
\draw (585.83,130.23) node [anchor=north west][inner sep=0.75pt]  [font=\tiny]  {$n_{1}$};
\draw (550.33,79.9) node [anchor=north west][inner sep=0.75pt]    {$\Gamma _{\{1,2,3\}}$};
\draw (552.33,155.9) node [anchor=north west][inner sep=0.75pt]  [font=\tiny]  {$\mathbb{P}^{1}$};
\draw (585.33,162.23) node [anchor=north west][inner sep=0.75pt]  [font=\tiny]  {$n_{2}$};
\draw (552.33,187.9) node [anchor=north west][inner sep=0.75pt]  [font=\tiny]  {$\mathbb{P}^{1}$};
\draw (585.83,194.23) node [anchor=north west][inner sep=0.75pt]  [font=\tiny]  {$n_{3}$};

\end{tikzpicture}

    \caption{An example of the four types of fixed loci. We depicted the case $\Adm{2,4,3,3}{6}$ to illustrate the most complicated situation that may happen: in the fixed locus $\Gamma_{\{1,2\}}$ the curves contracting over $0$ and $\infty$ may be disconnected even though the cover $E$ is connected.}
    \label{fig:fixedloci}
\end{figure}

Observe that if $4\in I$, then $\Gamma_I\cap ev_4^\ast(c_1(\Ocal_{\Pro^1}(1))) = \Gamma_I\cap ev_4^\ast(0)= \phi$, and therefore such fixed loci do not contribute to the localization computation. For all remaining fixed loci, we have $ev_4^\ast(0)_{|\Gamma_I}  = c_1(\Ocal_{\Pro^1}(1)_{|0})= t$.

\subsubsection*{Normal bundles.} The Euler class of the normal bundle $e(N_{\Gamma_I})$ has two types of contributions, that may be described in terms of the geometry of the base curves parameterized by the fixed locus,  see Figure \ref{fig:fixedloci}: if there is a marked point or a component contracting at $p$, one of the fixed points of $\Pro^1$, then a normal direction to the fixed locus may be identified with $T_p\Pro^1$. If a component $\tilde{C}$ of $C$ contracts at $p$, then another normal direction is identified with the  deformation space of the node; denoting by $\tilde{p}\in \tilde{C}$ the shadow of the node in the normalization of the curve, the deformation space of the node is described as $T_p\Pro^1\boxtimes T_{\tilde p} \tilde{C}$.

\subsubsection*{Restriction of the integrand to fixed loci.} The computation of the restriction of the class 
$c_2(R^1\pi_\ast f^\ast(\Ocal_{\Pro^1}(-1)\boxtimes \Ocal_e))$ to a fixed locus $\Gamma_I$ is slightly more sophisticated: the fibers of this bundle over a nodal cover $E\to B$ may be analyzed using the normalization sequence. Rather than attempting a general discussion here, we  carry out such analysis explicitly for each fixed locus. Refer to Figure \ref{fig:fixedloci} where each type of fixed locus is depicted.

\vspace{0.2cm}\noindent{$\mathbf{\Gamma_\phi} \cong \Adm{m_1,m_2, m_3, m_4, 0}{d}$}. A general point in this fixed locus corresponds to 
a cover $F:E\to C$, where 
$C = \Pro^1 \cup \tilde{C}$ is a nodal curve with exactly two components, one of which is mapped with degree $1$ to $\Pro^1$ (and therefore denoted $\Pro^1$), while the other contracts over $0$ . The cover $E$ consists of a connected cyclic cover of $E_0 \to \tilde{C}$ and $d$ copies of $\Pro^1$ mapping with degree $1$ to $\Pro^1$. Tensoring the normalization sequence
\begin{equation}\label{eq:normseqempty}
0\to \Ocal_{E}\to \Ocal_{E_0}\oplus \bigoplus_{i=1}^d \Ocal_{\Pro^1}\to \bigoplus_{i=1}^d \C_{n_i} \to 0    
\end{equation}
by the invertible sheaf $f^\ast(\Ocal_{\Pro^1}(-1)\boxtimes \Ocal_e)$ and taking the long exact sequence in cohomology, we obtain:
\begin{equation} \label{Femptygib}
0\to  L_0 \to H^1(E, F^\ast(\Ocal_{\Pro^1}(-1))_{d-e} \to H^1(E_0, \Ocal)_{d-e} \to 0.    
\end{equation}
Globalizing the fiberwise computation in \eqref{Femptygib}, we obtain:
\begin{equation}\label{Femptygibglob}
    R^1\pi_\ast f^\ast(\Ocal_{\Pro^1}(-1)\boxtimes \Ocal_e))_{|\Gamma_\phi} = \left((\bE^\vee)_{d-e}\oplus L_0 \right).
\end{equation}
The contribution of $\Gamma_\phi$ to the localization computation of \eqref{eq:auxint} is then:
\begin{equation}
\Cont(\Gamma_\phi) = \int_{\Gamma_\phi} \frac{t\  c_2\left((\bE^\vee)_{d-e}\oplus L_0 \right)} {t(t-\psi_0)} = 0 
\end{equation}
where the vanishing holds because the
bundle $(\bE^\vee)_{d-e}$ has rank $1$.



\vspace{0.2cm}\noindent{$\mathbf{\Gamma_{\{j\}}} \cong \Adm{m_1,m_2, m_3, m_4}{d}$}. For a general point in this fixed locus, the base  $C = \Pro^1 \cup \tilde{C}$ is a nodal curve with exactly two components, as in the previous case. The cover $E$ consists of a connected cyclic cover of $E_0 \to \tilde{C}$ and $q_j := \gcd(m_j, d)$ copies of $\Pro^1$ mapping with degree $r_j:= d/\gcd(m_j, d)$ to $\Pro^1$. 
The normalization sequence
\begin{equation}
0\to \Ocal_{E}\to \Ocal_{E_0}\oplus \bigoplus_{i=1}^{q_j} \Ocal_{\Pro^1}\to  \bigoplus_{i=1}^{q_j} \C_{n_i} \to 0    
\end{equation}
tensored by $f^\ast(\Ocal_{\Pro^1}(-1)\boxtimes \Ocal_e)$ gives rise to the long exact sequence in cohomology: 
\begin{equation} \label{Fonegib}
0\to \left(\bigoplus_{i=1}^{q_j} L_0\right)_{d-e} \to H^1(E, F^\ast(\Ocal_{\Pro^1}(-1)))_{d-e} \to H^1(E_0, \Ocal)_{d-e} 
\oplus \left(\bigoplus_{i=1}^{q_j} H^1(\Pro^1, \Ocal_{\Pro^1}(-r_j))
\right)_{d-e}
\to 0.
\end{equation}
One must consider two cases. If $d$ divides $q_je$, \eqref{Fonegib} reduces to:
\begin{equation} \label{Fonediv}
 0\to  L_0 \to H^1(E, F^\ast(\Ocal_{\Pro^1}(-1)))_{d-e} \to H^1(E_0, \Ocal)_{d-e} \to 0;   
\end{equation}
If $d\not|\ qe$, then we have:
\begin{equation}\label{Fonenotdiv}
 0\to  H^1(E, F^\ast(\Ocal_{\Pro^1}(-1)))_{d-e} \to H^1(E_0, \Ocal)_{d-e} 
\oplus L_{\left\langle\frac{em_j }{d}\right\rangle}
\to 0.
\end{equation}
In both cases one may deduce:

\begin{equation}
    R^1\pi_\ast f^\ast(\Ocal_{\Pro^1}(-1)\boxtimes \Ocal_e))_{|\Gamma_{\{j\}}} = \left((\bE^\vee)_{d-e} \right) \oplus L_{\left\langle\frac{em_j }{d}\right\rangle}.
\end{equation}
The contribution of $\Gamma_{\{j\}}$ to the localization computation of \eqref{eq:auxint} is then:
\begin{equation}
\Cont(\Gamma_{\{j\}}) = \int_{\Gamma_{\{j\}}} \frac{t\  c_2(\left((\bE^\vee)_{d-e} \right) \oplus L_{\left\langle\frac{em_j }{d}\right\rangle})} {-t^2(t-\psi_0)} = \frac{\left\langle\frac{m_j e}{d}\right\rangle}{t}\int_{\Adm{m_1,m_2, m_3, m_4}{d}} \lambda_1^e.
\end{equation}


\vspace{0.2cm}\noindent{$\mathbf{\Gamma_{\{1,2,3\}}} \cong \Adm{m_1,m_2, m_3, m_4}{d}$}. The analysis for this fixed locus is very similar to the previous case, the main difference being that the contracting component $E_\infty$ now lies over $\infty\in \Pro^1$.
Denoting $q_4 := \gcd(m_4, d)$ and $r_4:= d/\gcd(m_4, d)$, the long exact sequence in cohomology is then:
\begin{equation} \label{Fthreegib}
0\to \left(\bigoplus_{i=1}^{q_4} L_1\right)_{d-e}  \to H^1(E, F^\ast(\Ocal_{\Pro^1}(-1)))_{d-e} \to H^1(E_\infty, \Ocal)_{d-e}\otimes L_1 
\oplus \left(\bigoplus_{i=1}^{q_4} H^1(\Pro^1, \Ocal_{\Pro^1}(-r_4))
\right)_{d-e}
\to 0.
\end{equation}
As in the previous case one must analyse separately the cases when $d$ does or doesn't divide $q_4e$, but in both cases one may write:
\begin{equation}
    R^1\pi_\ast f^\ast(\Ocal_{\Pro^1}(-1)\boxtimes \Ocal_e))_{|\Gamma_{\{1,2,3\}}} = \left((\bE^\vee)_{d-e} \right)\otimes L_1 \oplus L_{1-\left\langle\frac{em_4 }{d}\right\rangle}.
\end{equation}
The contribution of $\Gamma_{\{1,2,3\}}$ to the localization computation of \eqref{eq:auxint} is then:
\begin{equation}
\Cont(\Gamma_{\{1,2,3\}}) = \int_{\Gamma_{\{1,2,3\}}} \frac{t\  c_2(\left((\bE^\vee)_{d-e} \right)\otimes L_1 \oplus L_{1-\left\langle\frac{m_4 e}{d}\right\rangle})} {t^2(t+\psi_\infty)} = \frac{\left\langle\frac{m_4 e}{d}\right\rangle-1}{t}\int_{\Adm{m_1,m_2, m_3, m_4}{d}} (\lambda_1^e +\psi_\infty).
\end{equation}


\vspace{0.2cm}\noindent{$\mathbf{\Gamma_{\{i,j\}}} \cong  d \cdot \Adm{m_k,m_4, 2d-m_k-m_4}{d}\times \Adm{m_i,m_j, d-m_i-m_j}{d} $}. These are zero dimensional fixed loci parameterizing covers with contracting components over both $0$ and $\infty$, so the only non-zero contributions in the localization computation arise from integrating classes which are multiples of some power of the equivariant parameter.
However, because we are in the numerical situation $\left\langle\frac{em_1}{d}\right\rangle +\left\langle\frac{em_4}{d}\right\rangle > 1$,  we can conclude from orbifold Riemann-Roch (Theorem \ref{thm:orroch}) that for
every choice of $k\in \{1,2,3\}$, $p_1^\ast(\bE^\vee)_{d-e}$ has rank one. Analyzing the long exact sequence in cohomology arising from tensoring the normalization sequence with $f^\ast(\Ocal_{\Pro^1}(-1)\boxtimes \Ocal_e)$, one sees that $p_1^\ast(\bE^\vee)_{d-e}$ is a summand for the rank two bundle $ R^1\pi_\ast f^\ast(\Ocal_{\Pro^1}(-1)\boxtimes \Ocal_e))_{|\Gamma_{\{i,j\}}}$. It follows that $c_2(R^1\pi_\ast f^\ast(\Ocal_{\Pro^1}(-1)\boxtimes \Ocal_e))_{|\Gamma_{\{i,j\}}})$ is a multiple of $\lambda_1^e$ and therefore it has no term which is a pure multiple of $t^2$, forcing $\Cont(\Gamma_{\{i,j\}}) = 0$.

\subsubsection*{Evaluation of auxiliary integral.} Adding all contributions (and ignoring the global factor of $1/t$), we obtain:
\begin{eqnarray}
0 & =& \sum_{I\in [3]} \Cont(\Gamma_I) \nonumber \\
 & =&  \int_{\Adm{m_1,m_2, m_3, m_4}{d}}  \left(\sum_{j=1}^3 \left\langle\frac{em_j}{d}\right\rangle \lambda_1^e\right) + \left(
\left\langle\frac{em_4 }{d}\right\rangle-1 \right)(\lambda_1^e +\psi_\infty) \nonumber \\
 &= & \left(  \int_{\Adm{m_1,m_2, m_3, m_4}{d}} \lambda_1^e \right) + \frac{\left\langle\frac{em_4 }{d}\right\rangle-1}{d},
\end{eqnarray}
where we have used that $\sum_{j=1}^4 \left\langle\frac{em_j}{d}\right\rangle = 2$ and that the degree of the class  $\psi_{\infty}$ is $1/d$. The result follows immediately solving for  the degree of $\lambda_1^e$.
\end{proof}

\section{The first Chern class of the full Hodge bundle $\bE$}
\label{sec:non-orb}

\subsection{The one dimensional case} In this section we prove Theorem \ref{thm:dimonenotorb}. While in principle one may use the relation
$$
\lambda_1 = \sum_{e=1}^{d-1}\lambda_1^e
$$
and Theorem \ref{thm:orbla} to compute the degree of $\lambda_1$, we found it more effective to use the Atiyah-Bott localization theorem. Before we begin the computation, we rephrase the statement of Theorem \ref{thm:dimonenotorb} in a less symmetric but more compact way.

\begin{corol}
The degree of $\lambda_1$ on the space $\Adm{m_1, m_2, m_3, m_4}{d}$ may be expressed  as:
\begin{equation}\label{eq:simpler}
\int_{\Adm{m_1, m_2, m_3, m_4}{d}}\lambda_1 = \frac{1}{12d^2} \left( d^2 - \sum_{i=1}^4{\gcd}^2\left(m_i, d\right)+\sum_{i=1}^3 {\gcd}^2\left(m_i+m_4, d\right)\right).
\end{equation}
\end{corol}
\begin{proof}
Since $m_1+m_2+m_3+m_4 = 0\in \Z/d\Z$ we have that  ${\gcd}\left(\sum_{i\in I}m_i, d\right) = {\gcd}\left(\sum_{i\in I^c}m_i, d\right)$. Formula \eqref{eq:simpler} is obtained from \eqref{eq:onedim} by taking only one representative for each pair $I, I^c$ and doubling its contribution. 
\end{proof}

We also specialize the theorem to the case of prime degree as the result is especially elegant in this case.
\begin{corol}
If $d = p$ is a prime number, then the degree of $\lambda_1$ on the space $\Adm{m_1, m_2, m_3, m_4}{p}$ is a rational function of $p$. Precisely:
\begin{equation}\label{eq:prime}
\int_{\Adm{m_1, m_2, m_3, m_4}{p}}\lambda_1 =
\left\{
\begin{array}{cl}
0 & \mbox{if}\  0\in \{m_1, m_2, m_3, m_4\} ,\\ & \\
\frac{p^2-1}{12p^2} & \mbox{if}\  m_i+m_j \not= 0\in \Z/p\Z \  \mbox{for all $i,j$},\\
& \\
\frac{p^2-1}{6p^2} & \mbox{if}\  \{m_1, m_2, m_3, m_4\}  = \{i, p-i, j, p-j\},\ \mbox{all distinct},\\ & \\
\frac{p^2-1}{4p^2} & \mbox{if}\  \{m_1, m_2, m_3, m_4\}  = \{i, p-i\}.
\end{array}
\right.
\end{equation}
\end{corol}

\begin{proof}[Proof of Theorem \ref{thm:dimonenotorb}]
Consider the auxiliary vanishing integral:
\begin{equation} \label{eq:auxintnotorb}
    \int_{\Adm{\Pro^1| m_1, m_2,m_3, m_4}{d}} ev_4^\ast(c_1(\Ocal_{\Pro^1}(1))) \cdot c_2(R^1\Pi_\ast F^\ast(\Ocal_{\Pro^1}(-1))) = 0;
\end{equation}
We adopt the same localization set-up as in Section \ref{loc:orb}, with the only difference that the line bundle $\Ocal_{\Pro^1}(-1)$ is pull-pushed via the morphisms from the universal cover curve $\Ucal_E$. The fixed loci, their normal bundles, and the restriction of the class $ev_4^\ast(c_1(\Ocal_{\Pro^1}(1)))$ are the same as in Section \ref{loc:orb}. Here we analyze the restrictions of the class $c_2(R^1\Pi_\ast F^\ast(\Ocal_{\Pro^1}(-1)))$. These computations have been coded and the code is available upon request.

\vspace{0.2cm}\noindent{$\mathbf{\Gamma_\phi} \cong \Adm{m_1,m_2, m_3, m_4, 0}{d}$}.  Tensoring the normalization sequence \eqref{eq:normseqempty}
by the invertible sheaf $F^\ast(\Ocal_{\Pro^1}(-1))$ and taking the long exact sequence in cohomology, we obtain:
\begin{equation} \label{Femptygibnotorb}
0\to  L_0 \to  L_0^{\oplus d} \to H^1(E, F^\ast(\Ocal_{\Pro^1}(-1))) \to H^1(E_0, \Ocal) \to 0.    
\end{equation}
Globalizing the fiberwise computation in \label{Femptygibglobnotorb}, we obtain:
\begin{equation}
   c_2( R^1\Pi_\ast F^\ast(\Ocal_{\Pro^1}(-1))_{|\Gamma_\phi}) = c_2(\left((\bE^\vee)\oplus \Ocal^{\oplus d-1}  \right)) = \lambda_2.
\end{equation}
The contribution of $\Gamma_\phi$ to the localization computation of \eqref{eq:auxintnotorb} is then:
\begin{equation}
\Cont(\Gamma_\phi) = \int_{\Gamma_\phi} \frac{t\  \lambda_2} {t(t-\psi_0)} = 0 
\end{equation}
where the vanishing holds by the projection formula because the
class $\lambda_2$ is obtained by pull-back of the homonymous class via the forgetful morphism $\pi_0:\Adm{m_1,m_2, m_3, m_4, 0}{d}\to \Adm{m_1,m_2, m_3, m_4}{d}$.


\vspace{0.2cm}\noindent{$\mathbf{\Gamma_{\{j\}}} \cong \Adm{m_1,m_2, m_3, m_4}{d}$}. Denoting $q_j = \gcd (m_j,d)$ the number of nodes over $0\in \Pro^1$ and $r_j = d/q_j$,  the relevant  long exact sequence in cohomology is: 
\begin{equation} \label{Fonegibnotorb}
0\to L_0\to \bigoplus_{i=1}^{q_j} L_0  \to H^1(E, F^\ast(\Ocal_{\Pro^1}(-1))) \to H^1(E_0, \Ocal) 
\oplus \bigoplus_{i=1}^{q_j} H^1(\Pro^1, \Ocal_{\Pro^1}(-r_j))
\to 0.
\end{equation}

Globally one observes that the bundles with fiber $ H^1(\Pro^1, \Ocal_{\Pro^1}(-r_j))$ are trivial, but not equivariantly trivial. Computing the torus weights (see \cite{r:adm}, for example)
one obtains
$$ H^1(\Pro^1, \Ocal_{\Pro^1}(-r_j)) = \bigoplus_{i=1}^{r_j-1} L_{\frac{i}{r_j}} $$
Globally, 
\begin{equation}\label{eq:fflnotorb}
    R^1\Pi_\ast F^\ast(\Ocal_{\Pro^1}(-1))_{|\Gamma_{\{j\}}} = (L_0^{\oplus q-1}\oplus \bE^\vee) \oplus \left(\bigoplus_{i=1}^{r_j-1} L_{\frac{i}{r_j}}\right)^{\oplus q_j}.
\end{equation}
The second Chern class of the bundle in \eqref{eq:fflnotorb} is computed using the formal properties of Chern classes from Section \ref{sec:chern}. Since $c^{eq}_t(L_0) = 1$ and we will be integrating on a one dimensional space, the only contributing part of $c_2$ is given by
\begin{equation}
  c_1(\bE^\vee) c_1\left(\left(\bigoplus_{i=1}^{r_j-1} L_{\frac{i}{r_j}}\right)^{\oplus q_j}\right)  + c_2\left(\left(\bigoplus_{i=1}^{r_j-1} L_{\frac{i}{r_j}}\right)^{\oplus q_j}\right).
\end{equation}
One computes
\begin{equation} \label{eq:weightsedgeterms}
 c_1\left(\bigoplus_{i=1}^{r_j-1} L_{\frac{i}{r_j}}\right)   = \frac{r_j-1}{2}t\ \  \ \ \  c_2\left(\bigoplus_{i=1}^{r_j-1} L_{\frac{i}{r_j}}\right) = \frac{(r_j-1)(r_j-2)(3r_j-1)}{24r_j} t^2.
\end{equation}
Using Lemma \ref{lem:chern} and the relation $qr=d$, one then evaluates \eqref{eq:fflnotorb} to:
\begin{equation}
  C_j:=-\frac{d-q_j}{2}\lambda_1 t +  \frac{(d - q_j)(3d^2 - 3dq_j - 4d + 2q_j)}{24d}t^2
\end{equation}

The contribution of $\Gamma_{\{i\}}$ to the localization computation of \eqref{eq:auxintnotorb} is then:
\begin{equation}
\Cont(\Gamma_{\{j\}}) = \int_{\Gamma_{\{j\}}} \frac{t\  C_j} {-t^2(t-\psi_0)} = \frac{1}{t}\left[\left(\frac{d-q_j}{2}\int_{\Adm{m_1,m_2, m_3, m_4}{d}} \lambda_1\right)-  \frac{(d - q_j)(3d^2 - 3dq_j - 4d + 2q_j)}{24d^2} \right].\end{equation}


\vspace{0.2cm}\noindent{$\mathbf{\Gamma_{\{1,2,3\}}} \cong \Adm{m_1,m_2, m_3, m_4}{d}$}. This fixed locus parameterizes curves with one component contracting over $\infty\in \Pro^1$. We denote $q_4 = \gcd(m_4,d)$  the number of nodes over $\infty$ and $r_4 = d/q_4$.
The long exact sequence in cohomology is:
\begin{equation} \label{Fthreegibnotorb}
0\to L_1 \to \bigoplus_{i=1}^{q_4}  L_1 \to H^1(E, F^\ast(\Ocal_{\Pro^1}(-1)) \to H^1(E_\infty, \Ocal)\otimes L_1 
\oplus \bigoplus_{i=1}^{q_4} H^1(\Pro^1, \Ocal_{\Pro^1}(-r_4))
\to 0.
\end{equation}
Globally one obtains:
\begin{equation} \label{eq:bundlenotorbthirdfixed}
    R^1\Pi_\ast F^\ast(\Ocal_{\Pro^1}(-1))_{|\Gamma_{\{1,2,3\}}} = (L_1^{ q_4-1}\oplus (\bE^\vee\otimes L_1)) \oplus \left(\bigoplus_{i=1}^{r_4-1} L_{\frac{i}{r_4}}\right)^{\oplus q_4}.
\end{equation}

The bundle $L_1^{ q_4-1}\oplus (\bE^\vee\otimes L_1) $ has $q_4-1$ Chern roots equal to $t$, plus the Chern roots $-\alpha_i+t$, where $\alpha_i$ is a Chern root of the Hodge bundle $\bE$. It follows that:
\begin{equation} \label{eq:coneandtwoforEtwited}
    c_1(L_1^{ q_4-1}\oplus \bE^\vee\otimes L_1) = -\lambda_1 +(g+q_4-1) t \ \ \ \  c_2(L_1^{ q_4-1}\oplus \bE^\vee\otimes L_1 ) = {{g+q_4-1}\choose{2}}t^2
-(g+q_4-2)t \lambda_1 +\lambda_2. \end{equation}

Combining \eqref{eq:weightsedgeterms} and \eqref{eq:coneandtwoforEtwited}, and neglecting the $\lambda_2$ term which will not survive integration, one has that the relevant part of the second Chern class of \eqref{eq:bundlenotorbthirdfixed} is 
\begin{equation}
    C_{1,2,3}:= 
    \frac{4-d-q_j-2g}{2}\lambda_1t+ \frac{3d^3 + 12gd^2 + 6d^2q_4 - 16d^2 + 12g^2d + 12gdq_4 + 3{q_4}^2d - 36gd - 18dq_4 + 24d - 2{q_4}^2}{24{d}}t^2
\end{equation}

The contribution of $\Gamma_{\{1,2,3\}}$ to the localization computation of \eqref{eq:auxintnotorb} is then:
\begin{equation*}
\Cont(\Gamma_{\{1,2,3\}}) = \int_{\Gamma_{\{1,2,3\}}} \frac{t\  C_{1,2,3}} {t^2(t+\psi_\infty)} = \frac{1}{t}\left[ \frac{4-d-q_4-2g
}{2}\int_{\Adm{m_1,m_2, m_3, m_4}{d}} \lambda_1^e \right.
\end{equation*}
\begin{equation}\left.
- \frac{3d^3 + 12gd^2 + 6d^2q_4 - 16d^2 + 12g^2d + 12gdq_4 + 3{q_4}^2d - 36gd - 18dq_4 + 24d - 2{q_4}^2}{24d^2}
\right].
\end{equation}


\vspace{0.2cm}\noindent{$\mathbf{\Gamma_{\{i,j\}}} \cong  d \cdot \Adm{m_k,m_4, 2d-m_k-m_4}{d}\times \Adm{m_i,m_j, d-m_i-m_j}{d} $}.
These are zero dimensional fixed loci parameterizing covers with contracting components over both $0$ and $\infty$. There are several discrete invariants associated to this fixed locus; we recall those that are used in the computation: there are $\gcd(m_i,m_j,d)$ components of the cover contracting to $\infty\in \Pro^1$; we denote by $g_{ij}$ the genus of the (possibly disconnected) cover contracting to $\infty\in\Pro^1$; we denote $q_{ij} = \gcd(m_i+m_j,d)$ the number of rational components of the cover mapping onto $\Pro^1$, and $r_{ij} = d/q_{ij}$. Since the only non-zero contributions in the localization computation arise from integrating classes which are multiples of some power of the equivariant parameter, we may simplify the normalization sequence by setting all geometric Chern roots of the Hodge bundles appearing to $0$, to obtain:

\begin{equation} \label{Fthreegibnotorb}
0\to L_1^{\gcd(m_i, m_j, d)} \to  L_1^{q_{ij}} \to H^1(E, F^\ast(\Ocal_{\Pro^1}(-1)) \to  L_1^{g_{ij} + \gcd(m_i, m_j, d)-1} 
\oplus \left(\bigoplus_{i=1}^{r-1} L_{\frac{i}{r_{ij}}}\right)^{q_{ij}}
\to 0.
\end{equation}

It follows that:

\begin{equation}\label{eq:fflnotorb}
    R^1\Pi_\ast F^\ast(\Ocal_{\Pro^1}(-1))_{|\Gamma_{\{i,j\}}} =L_1^{q_{ij}+g_{ij}-1}\oplus\left( \bigoplus_{i=1}^{r_{ij}-1} L_{\frac{i}{r_{ij}}}\right)^{\oplus q_{ij}}.
\end{equation}

One can observe that \eqref{eq:fflnotorb} formally agrees with \eqref{eq:bundlenotorbthirdfixed} upon substituting $g$ with $g_{ij}$ and $q_4$ with $q_{ij}$. The second Chern class of  $R^1\Pi_\ast F^\ast(\Ocal_{\Pro^1}(-1))_{|\Gamma_{\{i,j\}}}$ therefore may be read off from the $t^2$ coefficient of $C_{1,2,3}$.

The contribution of $\Gamma_{i,j}$ to the localization computation of \eqref{eq:auxintnotorb} is then:
\begin{equation}
\Cont(\Gamma_{i,j}) = \frac{3d^3 + 12g_{ij}d^2 + 6d^2q_{ij} - 16d^2 + 12g_{ij}^2d + 12g_{ij}dq_{ij} + 3{q_{ij}}^2d - 36g_{ij}d - 18dq_{ij} + 24d - 2{q_{ij}}^2}{24d^2}
\end{equation}

The theorem readily follows by solving  the equation
\begin{eqnarray}
0 & =& \sum_{I\in [3]} \Cont(\Gamma_I) \nonumber \\
\end{eqnarray}
for $\lambda_1$ and using the Riemann-Hurwitz formula to substitute 
$g = d+1 - (q_1+q_2+q_3+q_4)/2$ and $g_{ij} = 1- (q_i+q_j+q_{ij}-d)/2.$ 
\end{proof}

\subsection{Graph formula for $\lambda_1$}

In this section we prove Theorem \ref{thm:gf}, expressing the class $\lambda_1$ on a general space of cyclic admissible covers as a linear combination of boundary strata and of the classes $\psi_i$ and $\kappa_1$.

Since the group $A_1(\Adm{m_1, \ldots, m_n  }{d})$ is generated by classes of  boundary curves $C_{(X,Y,Z,W)}$,  to establish formula \eqref{eq:gf} it suffices to show the truth of the numerical equations following from intersecting with each boundary curve.

Choose a bijection $b:[4]\to \{X,Y,Z,W\}$. From Theorem \ref{thm:dimonenotorb}, it follows:
\begin{eqnarray}\label{eq:gff}
C_{(X,Y,Z,W)}\cdot \lambda_1 &= & \frac{1}{24 d^2} \left( \sum_{I\in \Pcal([4])}(-1)^{|I|} {\gcd}^2\left(\sum_{j\in \left(\bigcup_{i\in I} b(i)\right)} m_j, d\right)\right),
\end{eqnarray}
where we don't worry about reducing anything modulo $d$ as that operation is irrelevant when then taking a $\gcd$ with $d$ itself.

From standard boundary intersection theory in $\overline{M}_{0,n}$, together with the fact that $\Adm{m_1, \ldots, m_n  }{d}$ is a $B(\Z/d\Z)$ gerbe over
$\overline{M}_{0,n}$, one has:

\begin{equation}
 C_{(X,Y,Z,W)}\cdot \Delta_J = 
 \left\{
 \begin{array}{cl}
   \frac{(-1)^{|I|}}{d}   &  I\in \Pcal([4]), J = \bigcup_{i\in I} b(i),\\
     0 & \mbox{else}. 
 \end{array}
 \right.
\end{equation}

Observe  that the formula holds also when some of the sets $X,Y,Z,W$ are  singletons because of the definition $\Delta_{\{j\}} = \Delta_{[n]\smallsetminus\{j\}} = -\psi_j$. It is now immediate to verify that intersecting  $C_{(X,Y,Z,W)}$ with the left hand side of \eqref{eq:gf} produces the left hand side of \eqref{eq:gff},  concluding the proof of Theorem \ref{thm:gf}.

\bibliographystyle{alpha}
\bibliography{biblio}

\end{document}